\documentclass[11pt,a4,reqno,tbtags]{amsart}
\usepackage{amsmath}
\usepackage{amsfonts}
\usepackage{amssymb}
\usepackage{bbm}
\usepackage{latexsym}
\usepackage{graphicx}
\usepackage[utf8]{inputenc}
\usepackage{mathrsfs}

\usepackage{epstopdf}
\usepackage{hyperref}
\usepackage{url}
\usepackage{mdwlist}

\numberwithin{equation}{section}

\newtheorem{proposition}{Proposition}[section]

\newtheorem{theorem}[proposition]{Theorem}

\newtheorem{lemma}[proposition]{Lemma}

\theoremstyle{definition}

\newtheorem*{ack}{Acknowledgement}

\theoremstyle{remark}

\renewcommand\P{\mathbb{P}}
\newcommand{\PP}{\mathbb{P}}
\newcommand\E{\mathbb{E}}
\newcommand{\EE}{\mathbb{E}}
\newcommand\succg{\gamma}
\newcommand\geo{\gamma}
\newcommand\iv{s}
\newcommand{\Ma}{\mathcal{M}^\ast}
\newcommand{\dsub}{d^{\#}}
\newcommand{\rmap}{M}
\renewcommand{\l}{\lambda}
\newcommand{\oo}{\infty}
\newcommand{\se}{\subseteq}
\newcommand{\specdim}{d_{\mathrm{s}}}
\newcommand{\uiptroot}{\mathfrak{r}}
\newcommand{\sgt}{\mathcal{T}}
\newcommand{\Dec}{\mathcal{D}}

\newcommand{\indic}[1]{\mathbf{1}_{\{#1\}}}

\newcommand\cE{\mathcal{E}}

\newcommand\RR{\mathbb{R}}
\newcommand\Reff{\mathrm{R}^{\mathrm{eff}}}
\newcommand{\T}{\mathcal{T}}

\newcommand{\bea}{\begin{eqnarray}}
\newcommand{\eea}{\end{eqnarray}}

\def\void{}
\def\labelmark{}

\newenvironment{formula}[1]{\def\labelname{#1}
\ifx\void\labelname\def\junk{\begin{displaymath}}
\else\def\junk{\begin{equation}\label{\labelname}}\fi\junk}%
{\ifx\void\labelname\def\junk{\end{displaymath}}
\else\def\junk{\end{equation}}\fi\junk\labelmark\def\labelname{}}

{\ifx\void\labelname\def\junk{\end{array}\end{displaymath}}
\else\def\junk{\end{array}\right.\end{equation}}
\fi\junk\labelmark\def\labelname{}\def\junk{}
}

\newcommand{\beq}{\begin{formula}}
\newcommand{\eeq}{\end{formula}}
\newcommand{\beqv}{\begin{formula}{}}

\newenvironment{romenumerate}[1][0pt]{
\addtolength{\leftmargini}{#1}\begin{enumerate}
 }{\end{enumerate}}

\newcounter{oldenumi}
{\setcounter{oldenumi}{\value{enumi}}
\begin{romenumerate} \setcounter{enumi}{\value{oldenumi}}}
{\end{romenumerate}}

\begingroup
  \divide\count255 by 60
  \multiply\count255 by -60
  \advance\count255 by \time
  \ifnum \count255 < 10 \xdef\klockan{\the\count1.0\the\count255}
  \else\xdef\klockan{\the\count1.\the\count255}\fi
\endgroup

\newcommand\marginal[1]{\marginpar{\raggedright\parindent=0pt\tiny #1}}

\newcommand\REM[1]{{\raggedright\texttt{[#1]}\par\marginal{XXX}}}

\newcommand\gb{\beta}
\newcommand\gd{\delta}

\newcommand\eps{\varepsilon}

\def\rompar(#1){\textup(#1\textup)}

\def\xexp(#1){e^{#1}}

\newcommand\eqd{\overset{\mathrm{d}}{=}}

\newcommand\ZZ{\mathbb Z}

\newcounter{CC}

\newcounter{cc}

\newcommand{\out}{\text{out}}

\newcommand{\map}{\mathbf{m}}

\allowdisplaybreaks

\begin{document}
\title
{Recurrence of bipartite planar maps}

\date{\today}  

\thanks{Research supported by the Knut and Alice Wallenberg Foundation.}

\author{Jakob E. Bj\"ornberg}
\author{Sigurdur \"Orn Stef\'ansson} 

\address{Department of Mathematics, Uppsala University, PO Box 480,
SE-751~06 Uppsala, Sweden}

\maketitle

\begin{abstract} 

This paper concerns random bipartite planar maps which are defined by
assigning weights to their faces. The paper
presents a threefold contribution to the theory.  Firstly, we prove
the existence of the local limit for all choices of weights and
describe it in terms of an infinite mobile.  Secondly, we show that
the local limit is in all cases almost surely recurrent.  And thirdly,
we show that for certain choices of weights the local limit has
exactly one face of infinite degree and has in that case spectral
dimension $4/3$ (the latter requires a mild moment condition).
\end{abstract}

\section{Introduction and main results}

A planar map is
a finite connected graph
embedded in the 2-sphere.  Motivated by questions of universality, 
Marckert and Miermont~\cite{marckert:2007} 
introduced a class of probability 
distributions on bipartite planar maps, where a weight
$q_{d/2}$ is given to each face of degree $d$.  
(A bipartite map is one in which all faces have even degree,
hence $d/2$ is an integer.)
Due to the
discovery of certain bijections between planar maps and 
labeled trees \cite{bouttier:2004,schaeffer:thesis}
progress on this model of random planar maps has been
tremendous, see e.g. \cite{legall:notes} for a recent review.  Much of the focus
has been on the \emph{scaling limit} where the map is 
rescaled by some power of its size and one studies
the limit in the Gromov--Hausdorff sense of 
the corresponding metric space, see 
e.g.~\cite{janson:2012sl,legall:2007,legall:2011,legall:unique,miermont:unique}.
Other papers focus on the \emph{local limit}, 
where one does not rescale the graph and
the limiting object, when it exists, is 
an infinite graph, see e.g. \cite{angel:2003, chassaing:2006,  curien:2012a, krikun:2005, menard:2013} for results on special cases and related models.  So far the local limit
has been shown to exist only for certain choices
of the weights $q_i$.

This paper presents three main contributions to the theory
of local limits of planar maps
(precise definitions and statements appear in the following
subsections).  Our first main
result is the existence, and a description, of the local
limit for arbitrary choices of the weigths $q_i$,
see Theorem~\ref{thm:locallimit}.  
 We show this by using a connection to
simply generated trees, and a recent general limit
theorem due to Janson for the latter object \cite{janson:2012sgt}. The approach is similar to the one of Chassaing and Durhuus \cite{chassaing:2006} and later Curien, Ménard and Miermont \cite{curien:2012a} in the case of quadrangulations.
Our second main result is that the limit map is almost surely
recurrent for all choices of the weights $q_i$, see
Theorem~\ref{thm:recurrence}.  This is proved using a recent general
result of Gurel-Gurevich and Nachmias \cite{gurel:2013}, 
and relies on establishing exponential tails of the
degrees in the local limit (Theorem~\ref{deg_thm}).
Our third main result
 focuses on finer properties of random walk in
the local limit, for parameters in a certain 
`condensation phase'.  In this phase, the local limit almost surely
has a face of infinite degree.  Under an additional moment
condition, we show that the \emph{spectral dimension}
of the map in this case is almost surely $4/3$
(see Theorem~\ref{th:spec}).  
Roughly speaking this follows from the fact that, from the point of view
of a simple random walk, the map is tree-like
(although it is in general \emph{not} a tree).  This result relies on
recent general methods for expressing the spectral dimension
in terms of volume and resistance growth
due to Kumagai and Misumi \cite{kumagai:2008}.  We now define
the model more precisely.

\subsection{Planar maps} \label{s:planarmaps}

A planar map is a finite, connected graph embedded in the 2-sphere
and viewed up to orientation preserving homeomorphisms of the
sphere. A connected component of the complement of the edges of the
graph is called a \emph{face}. The \emph{degree} 
of a face $f$, denoted by $\deg(f)$
is the number of edges in its boundary;  
the edges are counted with multiplicity,
meaning that the same edge is counted twice if both its sides are
incident to the face.  We sometimes
consider \emph{rooted} and \emph{pointed} planar maps: the
\emph{root} is then a distinguished oriented edge $e = (e_{-},e_+)$,
and the  \emph{point} is a fixed marked  vertex, which will be denoted by
$\rho$.  All maps we consider are bipartite; this is
equivalent to each face having an even degree. 
We denote the set of finite
bipartite, rooted and pointed planar maps by
$\mathcal{M}_\mathrm{f}^\ast$, and we denote the subset of maps 
with $n$ edges by $\mathcal{M}_n^\ast$. 
For a planar map $\map$ we denote the set of
vertices, edges and faces by $V(\map)$, $E(\map)$ and $F(\map)$
respectively.
For a map $\map\in \mathcal{M}_\mathrm{f}^\ast$ and an integer $r\geq 0$, 
let $B_r(\map)$ denote the
planar subgraph of $\map$ spanned by the set of vertices at a graph
distance $\leq r$ from the origin $e_{-}$ of the root edge.  
Note that $B_r(\map)$ is a planar map;  for $r\geq 1$ it
is rooted, and it is pointed if the vertex $\rho$ is at a graph
distance $\leq r$ from $e_{-}$.  Define a metric on
$\mathcal{M}_\mathrm{f}^\ast$ by
\begin {equation}
d_\mathcal{M}(\map_1,\map_2) = \left(1+\sup\left\{r :
B_r(\mathbf{m_1})=B_r(\mathbf{m_2})\right\}\right)^{-1},\quad
\map_1,\map_2\in \mathcal{M}_\mathrm{f}^\ast.
\end {equation}
This metric on rooted graphs was introduced 
in~\cite{benjamini:2001}.
Denote by $\mathcal{M}^\ast$ the completion of $\mathcal{M}_\mathrm{f}^\ast$
with respect to $d_\mathcal{M}$.  Thus $\mathcal{M}^\ast$ is a metric
space, which we further make into a measure space by 
equipping it with the Borel $\sigma$-algebra. The elements
of $\mathcal{M}^\ast$ which are not finite are called \emph{infinite} planar
maps and the set of infinite planar maps is denoted by $\Ma_\infty$. 
An infinite planar map $\map$ can be represented by an equivalence class of
sequences $(\map_i)_{i\geq 0}$ of finite planar maps having the property that for each $r\geq 0$,
$B_r(\map_i)$ is eventually the same constant for every
representative.
We then call $\map$ the \emph{local limit} of the sequence
$(\map_i)_{i\geq 0}$.
The equivalence class defines a
unique infinite rooted graph, which may or may not be
pointed. 

We will consider probability measures on $\Ma$ which are
defined via a sequence $(q_i)_{i\geq 1}$ of non-negative numbers, as follows.
Define a sequence of probability measures $(\mu_n)_{n\geq 1}$ on $\Ma$
by first assigning to each finite map $\map$ a weight
\begin {equation}
 W(\map) = \prod_{f\in F(\map)} q_{\deg(f)/2}
\end {equation}
and setting
\begin {equation}\label{main_def}
 \mu_n(\map) = \frac{W(\map)}{\sum_{\map'\in\Ma_n}W(\map')},
\mbox{ if } \map\in\Ma_n \mbox{ (otherwise 0)}.
\end {equation}
This definition was first introduced by
Marckert and Miermont~\cite{marckert:2007}.

\subsection{Main results} \label{ss:mainresults}
Our first main result
establishes a weak limit of $(\mu_n)_{n\geq1}$ in the topology
generated by $d_\mathcal{M}$. In order to exclude the trivial case
when all faces have degree two we demand that $q_i > 0$ for some
$i\geq 2$. Certain qualitative properties of the limit map can 
be determined by the value of a quantity
$\kappa$ which we will now define.  For convenience, we will define a new sequence $(w_i)_{i\geq 0}$, expressed in terms of the parameters $(q_i)_{i\geq 1}$ as
\begin {equation} \label{wqdef} 
 w_i = \binom{2i-1}{i-1} q_i, \quad \text{for} ~i\geq 1
\end {equation}
and we let $w_0  = 1$. The reason for this definition 
will become clear when we explain the connection between 
the maps and simply generated trees in Section \ref{s:sgt}. 
Denote the generating function of $(w_i)_{i\geq 0}$ by
\begin{equation}
g(z) = \sum_{i=0}^\infty w_i z^i
\end{equation}
 and denote its radius of convergence by $R$. 
If $R > 0$ define 
\begin {equation} \label{defgamma}
 \gamma = \lim_{t\nearrow R} \frac{tg'(t)}{g(t)}
\end {equation}
and if $R = 0$ let $\gamma = 0$. 
Note that the ratio in~\eqref{defgamma} is continuous and increasing
in $t$ by~\cite[Lemma~3.1]{janson:2012sgt}.
Next define the number $\tau \geq 0$ to be
\begin {enumerate}
\item the (unique) solution $t\in(0,R]$ to $tg'(t)/g(t) = 1$, if $\gamma \geq 1$; or 
\item $\tau = R$, if $\gamma < 1$.
\end {enumerate}
Then define the probability weight sequence $(\pi_i)_{i\geq 0}$ by
\begin {equation} \label{pi}
\pi_i = \frac{\tau^i w_i}{g(\tau)}
\end {equation}
and let $\xi$ be a random variable distributed by 
$(\pi_i)_{i\geq 0}$.  We will view $\xi$ as an 
offspring distribution of a Galton--Watson process 
and we denote its expected value by $\kappa$. One easily finds that
\begin {equation}
 \kappa = \min\{\gamma,1\} \leq 1
\end {equation}
i.e.~$\xi$ is either critical or sub-critical. These definitions come
from~\cite[Theorem~7.1]{janson:2012sgt}, restated below as
Theorem~\ref{th:janson}.

Here is a brief remark about how our notation relates to
that of~\cite{marckert:2007},
see also~\cite[Appendix~A]{janson:2012sl}
for a discussion about this.  Instead of our generating
function $g$, Marckert and Miermont 
consider $f=f_{\mathbf{q}}$ given by
$f(z)=\sum_{k\geq0}w_{k+1}z^k$.  Thus $g(z)=1+zf(z)$,
and the radius of convergence $R$ of $g$ equals that of $f$
(called $R_{\mathbf{q}}$ in~\cite{marckert:2007}).
In~\cite{marckert:2007} the weights $(q_i)_{i\geq 1}$ are
called \emph{admissible} if there is a solution 
$z\in(0,\infty)$ to $f(z)=1-1/z$, which in our notation
becomes $z=g(z)$.  In the case $\gamma\geq1$ we have
$\tau g'(\tau)=g(\tau)$, which is equivalent to $\tau^2f'(\tau)=1$,
the form used in~\cite{marckert:2007}.  In this 
case Marckert and Miermont refer to the weights $(q_i)_{i\geq1}$
as \emph{critical}, and their $Z_{\mathbf{q}}$ equals
our $\tau$.

Before stating the first theorem we recall some definitions. Firstly
an infinite graph is said to be one-ended if the complement of every
finite connected subgraph
contains exactly one infinite
connected component. Secondly, we recall the construction of the uniform
infinite plane tree (UIPTree). 
It is the infinite plane tree
distributed by the weak local limit, as $n\rightarrow \infty$, of the
uniform measures on the set of all plane trees with $n$ edges. An
explicit construction
(when $w_i =1$ for all $i$) is given in Section \ref{s:sgt}. The UIPTree can
also be viewed as the weak local limit of a Galton--Watson tree with
offspring distribution $2^{-i-1}$ conditioned to survive.  
\begin{theorem}\label{thm:locallimit}
For all choices of weights $(q_i)_{i\geq 1}$ 
the measures $\mu_n$ converge
weakly to some probability measure $\mu$ (in the topology
generated by $d_\mathcal{M}$).  The infinite map with
distribution $\mu$ is almost surely one-ended
and locally finite.
If $\kappa = 1$  all faces
are of finite degree. 
If $\kappa < 1$ the map 
contains exactly one face of infinite
degree. 
If $\kappa = 0$ the limit is the UIPTree.
\end{theorem}

Special cases of this result have been established previously,
see \cite{chassaing:2006,curien:2012a,krikun:2005} for the case of uniform quadrangulations. 
Other related results in the case of non-bipartite graphs have been established for uniform triangulations \cite{angel:2003}  and uniform maps \cite{menard:2013}.

The proof of Theorem \ref{thm:locallimit} 
appears at the end of Section~\ref{s:trees} and
relies on two bijections: first a bijection
due to Bouttier, Di Francesco and Guitter (BDG) \cite{bouttier:2004}
 from $\mathcal{M}_{\mathrm{f}}^\ast$ to a class of labelled trees 
called \emph{mobiles}, and then a bijection which maps random 
mobiles of the form we consider to \emph{simply generated trees}.
In \cite{janson:2012sgt}, Janson established a general convergence result
for simply generated trees, which allows us to deduce the 
corresponding convergence result for planar maps.  This 
correspondence was previously used by Janson and Stefánsson 
to study scaling limits
of planar maps in \cite{janson:2012sl}. 
We note here that one may deduce more details about the structure
of the limiting map than those stated in
Theorem \ref{thm:locallimit} from the upcoming
Theorem~\ref{th:tildenu}.  The latter result concerns the 
local limit of the mobiles along with the procedure of 
constructing the maps out of the mobiles.  
The local limit of the mobiles has  an
explicit description in terms of a multi-type 
Galton--Watson process as will be explained in 
Section~\ref{ss:tildenu}. We expect that, similarly
to the case of quadrangulations~\cite{curien:2012a},
one may recover the infinite mobile from the infinite map.
We remark here that the bipartite case, on which we focus,
is easier since the BDG bijection has a particularly simple 
form in this case, which directly gives the correspondence
with simply generated trees.  In~\cite{bouttier:2004}
bijections between trees and more general types of
maps are described, see also~\cite{miermont2006invariance}.

Throughout the paper, we will let $M$ denote the infinite random map
distributed by $\mu$ from Theorem \ref{thm:locallimit}. 
Recall
that simple random walk on a locally finite graph
$G$ is a Markov chain starting at some specified vertex,
which at each integer time step  moves to a uniformly chosen neighbour.
Recall also that $G$ is \emph{recurrent}
if simple random walk returns to its starting point
with probability 1.  Our second
main result is the following,
and is proved in Section~\ref{s:rec}:
\begin{theorem}\label{thm:recurrence}
For all choices of the weights $q_i$,
the map $M$ is almost surely recurrent.
\end{theorem}
The proof relies on a recent result  on recurrence
of local limits~\cite{gurel:2013}.  To be able to apply this result we show that 
the degree of a typical vertex in $M$ has an exponential tail,
see Theorem~\ref{deg_thm}.  We note here that the degree
of a typical \emph{face} need not have exponential tails
(even in the case $\kappa=1$).

Our third main result
 concerns the case $\kappa<1$, when $M$ has a
unique face of infinite degree. This phase has been referred to as the
condensation phase in the corresponding models of simply generated
trees \cite{janson:2012sgt,jonsson:2011}. Our results concern the
asymptotic behaviour of the return probability of a simple random walk
after a large number of steps. 
Let $p_G(n)$ be the probability that simple
random walk on $G$ is back at its starting point 
after $n$ steps. The {\it spectral dimension} of $G$ is defined
as 
\begin {equation}
 \specdim(G) = -2\lim_{n\rightarrow \infty} 
\frac{\log(p_G(2n))}{\log n}
\end {equation}
provided the limit exists. It is simple to check that the limit is
independent of the initial location of the walk if $G$ is
connected. Let $\xi$ be the random variable defined below (\ref{pi})
and recall that $\kappa=\EE(\xi)$.
\begin {theorem}\label{th:spec}
If $\kappa< 1$ and if there exists $\beta > 5/2$
such that $\mathbb{E}(\xi^\beta) < \infty$ then 
almost surely $\specdim(M) = 4/3$.
\end {theorem}
Recall that $M$ is the UIPTree when $\kappa=0$.
It was shown in \cite{fuji:2008} (see also
\cite{barlow:2006,durhuus:2007,kesten:1986}) 
that the spectral dimension of the UIPTree (and even
the more general class of critical Galton--Watson trees with finite
variance conditioned to survive) is almost surely 4/3. Our results in
the case $\kappa = 0$ are therefore in agreement with those
results.  
When $0<\kappa < 1$ the map $M$ is however no longer a tree
but we show that from the point of view of the random walk it is still
tree--like
(see also~\cite{bettinelli} for the phenomenon that 
maps with a unique large face are tree-like). 
This is perhaps not surprising in view of recent results
in \cite{janson:2012sl}, where it is shown, for regular
enough weights, that as $n\to \infty$ the scaling 
limit of the maps, with the
graph metric rescaled by $n^{-1/2}$, 
is a multiple of Aldous' Brownian tree. In this sense, the maps are
globally tree like although they contain a number of small loops, see
Fig.~\ref{f:sim} for an example.  The condition $\beta>5/2$
is probably not optimal but is the best that can
be obtained with our methods.  We suspect that $\beta>1$ suffices.

\begin{figure} [t]
\centerline{\scalebox{0.4}{\includegraphics{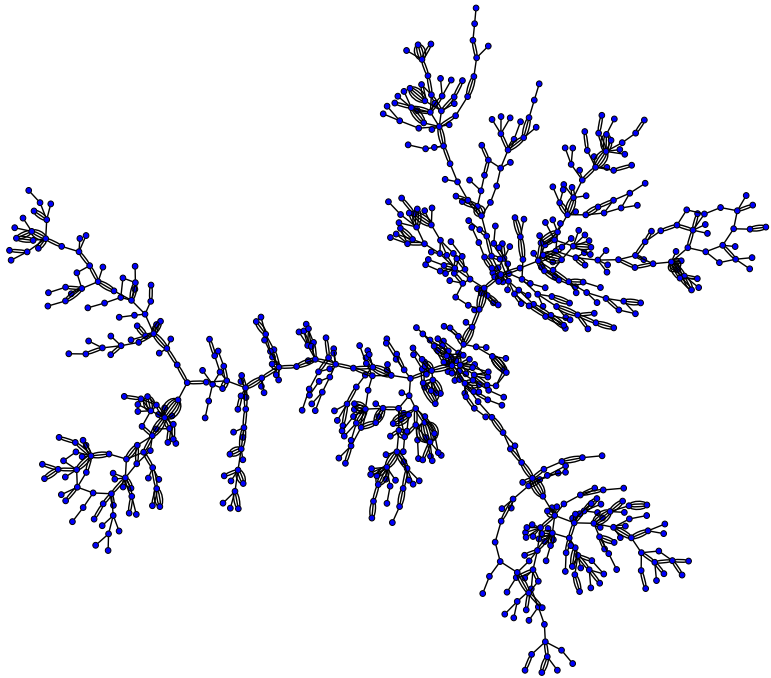}}}
\caption{A simulation of a planar map with $\kappa = 0.66$ and $w_i
  \sim i^{-3}$. The map has 794 vertices and 579 faces. The drawing is non-isometric and non-proper.} \label{f:sim}
\end{figure}

It is worth noting here that the value 4/3 for the spectral dimension
is also encountered in critical percolation clusters of $\mathbb{Z}^d$
for $d$ large enough.  It was conjectured by 
Alexander and Orbach~\cite{alexander:1982} that the spectral 
dimension of the incipient infinite cluster
for percolation on $\ZZ^d$ ($d\geq2$) should be 4/3.  This conjecture
is generally  believed to be true for $d > 6$ 
(but false for $d\leq 6$)
and has been proven for $d \geq 19$ and 
for $d>6$ when the lattice is sufficiently spread out \cite{kozma:2009}.

Finally, Theorem~\ref{th:spec} may be seen as a refinement of
Theorem~\ref{thm:recurrence} (for $\kappa<1$) in the sense that
if a graph $G$ is recurrent, and the spectral dimension $\specdim(G)$
exists, then $\specdim(G)\leq2$.  We do not prove the existence of $\specdim(M)$
other than in the case covered by Theorem~\ref{th:spec}.

\subsection{Outline}
The paper is organized as follows. In Section \ref{s:trees} we
introduce rooted plane trees and mobiles and explain how they may be related
to planar maps via the BDG bijection. In Section~\ref{s:conv} we
prove Theorem~\ref{thm:locallimit} on the
existence and characterization of the local limit,
and in Section~\ref{s:rec} we prove 
Theorem~\ref{thm:recurrence} on recurrence.
Section~\ref{s:rw} is devoted to 
the spectral dimension  in the condensation
phase when $\kappa <1$ (Theorem~\ref{th:spec}). 
In order to improve the readability of the main text
 we collect proofs of some lemmas 
in the Appendix.

\section{Trees and mobiles} \label{s:trees}

The study of planar maps is intimately tied up with the study of
trees, as will be explained in the following sections.  In this
section we introduce our main definitions and tools for studying
trees.  As a `reference' we will use a certain infinite tree
$T_\infty$, whose vertex set is $V(T_\infty) = \bigcup_n \mathbb{Z}^n$ 
i.e.~the set of all finite sequences of
integers.  The tree $T_\infty$ is closely related to the standard
Ul\'am--Harris tree (which has vertex set $\bigcup_n \mathbb{N}^n$),
and is defined as follows.  
Firstly,
the \emph{concatenation} of two elements $u,v \in V(T_\infty)$ is
denoted by $uv$.
The unique vertex in $\mathbb{Z}^0$ (the empty
sequence) is called the \emph{root}
(not to be confused with the root edge of a map)
and is denoted by $\varnothing$. The
edges in $T_\infty$ are defined by connecting every vertex $vi$, $i\in
\mathbb{Z}$, $v\in V(T_\infty)$, to the corresponding vertex $v$. In
this case $v$ is said to be the parent of $vi$ and $vi$ is said to be
the child of $v$. More generally, $v$ is said to be an ancestor of
$v'$ if $v' = vu$ for some $u\in V(T_\infty)$ and in that case $v'$ is
said to be a descendant of $v$. We denote the genealogical relation by
$\prec$ i.e.~$v \prec v'$ if and only if $v$ is an ancestor of
$v'$. The generation of $v$ is defined as the number of elements in
the sequence $v$ or equivalently as the graph distance of $v$ from the
root, and is denoted by $|v|$.

\subsection{Rooted plane trees}\label{ss:pt}
A rooted, plane tree $T$, with vertex set $V(T)$, is defined as a
subtree of $T_\infty$ containing the root $\varnothing$ and having the
following property: For every vertex $v\in V(T)$ there is a number
$\out(v) \in \{0,1,\ldots\}\cup\{\infty\}$, called the \emph{outdegree} of
$v$, such that $vi \in V(T)$ if and only if $\lfloor -\out(v)/2\rfloor
< i \leq \lfloor \out(v)/2\rfloor$, see Fig.~\ref{f:planar}. The
degree of a vertex $v$ is denoted by $\deg(v)$ and defined as $\deg(v)
= \out(v)+1$ if $v\neq \varnothing$ and $\deg(\varnothing) =
\out(\varnothing)$. For each vertex $v$ we order its children from
left to right by declaring that $vi$ is to the left of $vj$ if
\begin {itemize}
\item $i=0$ ($v0$ is the leftmost child) or
\item $ij > 0$ and $i<j$ or
\item $i > 0$ and $j<0$.
 \end {itemize}
Our definition of a rooted plane tree is equivalent to the
conventional definition (see e.g.~\cite{janson:2012sgt}) if the tree is locally
finite, i.e.~$\out(v) < \infty$ for all $v\in V(T)$. However, it
differs slightly if the tree has a vertex $v$ of infinite degree since
in our case $v$ both has a leftmost and a rightmost child whereas
conventionally it would only have a leftmost child.  It will be
important to have this property when describing planar maps in the
so-called condensation phase. 
All trees we consider in this paper will be plane trees and we will from now on simply refer to them as trees.

\begin{figure} [t]
\centerline{\scalebox{1.1}{\includegraphics{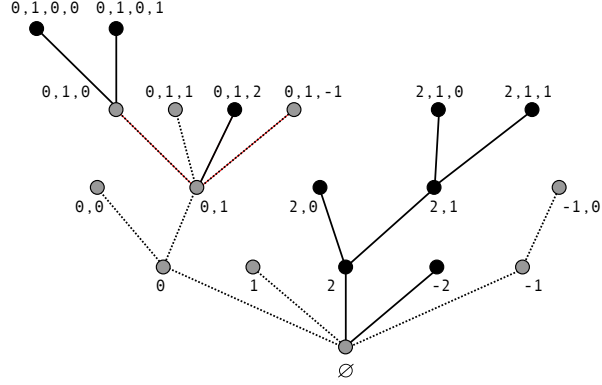}}}
\caption{An example of a plane tree $T$. The subtree $T^{[3]}$ is indicated by dashed edges and gray vertices.} \label{f:planar}
\end{figure}

Denote the set of trees with $n$ edges by $\Gamma_n$ and the set of 
all finite trees by $\Gamma_\mathrm{f}=\bigcup_n \Gamma_n$. 
In this paper we will only consider infinite trees $T$ which 
have either of the two properties:
\begin {enumerate}
\item $T$ is locally finite and there is exactly one infinite self-avoiding path starting at the root 
called an {\it infinite spine};  or
\item Exactly one vertex in $T$ has infinite degree and $T$ contains no infinite spines.  
The unique self-avoiding path from the root to 
the vertex of infinite degree is in this case called a {\it finite spine}.   
\end {enumerate}
Denote the set of such infinite trees by $\Gamma_\infty$ and let 
$\Gamma = \Gamma_\mathrm{f}\cup\Gamma_\infty$. 
A tree satisfying (1) or (2) can 
be embedded in the plane in such a way that all vertices are isolated
points, no edges cross and so that the ordering of its vertices is
preserved. When we refer to embeddings of trees later on we
will always assume that these properties hold.  

When an infinite tree
has a spine (finite or infinite, as above) 
we will denote the sequence of vertices on the spine, 
ordered by increasing distance from the root, by 
$\varnothing = S_0,S_1,\ldots$. When there is a vertex of 
infinite degree we will denote it by $\iv$ and we will denote its 
children by $\iv_i$, $i\in \mathbb{Z}$ orderered from left to right
in the  same way as before.

One may define a metric on $\Gamma$ in much the same way as
we did for $\Ma$, as follows.
For every $R\geq 0$ define the set 
\begin {equation}
V^{[R]} = \bigcup_{n=0}^R \{\lfloor -R/2\rfloor+1,\lfloor -R/2\rfloor+2,\ldots,\lfloor R/2\rfloor-1,\lfloor R/2\rfloor\}^n
\end {equation}
and for $T\in\Gamma$ let $T^{[R]}$ be the finite subtree of $T$ with vertex set $V(T)\cap V^{[R]}$, see Fig.~\ref{f:planar}. Define the metric
\begin {equation}
 d_\Gamma(T_1,T_2) =  \left(1+\sup\left\{R~:~T_1^{[R]}=T_2^{[R]}\right\}\right)^{-1},\qquad T_1,T_2\in\Gamma. 
\end {equation}
The set $\Gamma$ is equipped with the Borel $\sigma$-algebra generated by $d_\Gamma$.

\subsection{Mobiles} \label{ss:mobiles} 
It will be convenient to emphasise the 
distinction between vertices in a tree that belong to odd and
even generations, respectively.
For each tree $T\in\Gamma$ we therefore colour the root and vertices in every
even generation \emph{white} and we colour vertices in every odd generation 
\emph{black}.   The set of black (resp.~white) vertices in the tree $T$
will be denoted by $V^\bullet(T)$ (resp.~$V^\circ(T)$).
 Let $\Gamma_\infty^\odot$ be the subset of
$\Gamma_\infty$
where only black vertices can have infinite degree and define
$\Gamma^\odot = \Gamma_\mathrm{f}\cup\Gamma_\infty^\odot$.

For a finite tree $T \in \Gamma_n$, define the \emph{left contour 
sequence} $(c_i^{(L)})_{i\geq 0}$ of vertices in $T$ as follows:
\begin {itemize}
 \item $c_0^{(L)} = \varnothing$,
 \item For each $j < 2n$, the element in  $(c_i^{(L)})_{i\geq 0}$ 
following $c_j^{(L)}$ is the leftmost child of $c_j^{(L)}$ 
which has still not appeared in the sequence or if all 
its children have appeared it is the parent of $c_j^{(L)}$.
 \item The sequence is extended to $i> 2n$ by $2n$ periodicity.
\end {itemize}
 Similarly define the \emph{right contour sequence} $(c_i^{(R)})_{i\geq 0}$ 
by replacing {\it leftmost} with {\it rightmost} in the above 
definition. Next, define the contour 
sequence $(c_i)_{i\in\mathbb{Z}}$ by
\begin {equation} \label{contourseq}
 c_i = \left\{\begin {array} {ll}
c_i^{(L)} & \text{if}~i \geq 0 \\
c_{-i}^{(R)} & \text{if}~i < 0.
\end {array}\right.
\end {equation}
We will refer to each occurrence of a vertex $v$ in the contour 
sequence as a \emph{corner} of $v$.  Note that $v$ has $\deg(v)$ number 
of corners.  We extend the above definitions to elements 
in $\Gamma_\infty$ in the
obvious way (there is only one infinite period);  this is possible due
to how the infinite trees are constructed and how the children of the
vertex of infinite degree are ordered.  
Note that for a tree $T$ in $\Gamma$, the contour
sequence visits all vertices.
We will sometimes use the
term \emph{clockwise} (respectively, \emph{counterclockwise})
contour sequence, which refers to progressing through the
contour sequence $c_i$ by increasing (respectively, decreasing) 
the index $i$.

Define the white contour
sequence $(c^\circ_i)_{i\in\mathbb{Z}}$ by $c^\circ_i = c_{2i}$ for
all $i\in\mathbb{Z}$. Note that every white vertex appears in this
sequence. Similarly, for a tree with a
(finite or infinite) spine let $(S^\circ_i)_{i\geq0}$ be 
a sequence of the white vertices on the spine defined by
$S^\circ_i = S_{2i}$.

For trees $T\in\Gamma^\odot$, we will consider integer labels 
$(\ell(v))_{v\in V^\circ(T)}$ assigned to the white vertices of 
$T$, and which obey the following rules.
\begin {enumerate}
\item \label{e:labelinc} For all $i\in\mathbb{Z}$,
  $\ell(c^\circ_{i+1}) \geq \ell(c^\circ_{i})-1$ (for every black
  vertex $u$, the labels of the white vertices adjacent to $u$ can
  decrease by at most one in the clockwise order around $u$).
\item If $T$ has an infinite spine then $\inf_{i\geq 0}\ell(S^\circ_i) =
  -\infty$.
\item \label{it:infdegree} If $T$ has a vertex of infinite degree then\\ $\inf_{i\geq
  0}\ell(\iv_i) = \inf_{i<0}\ell(\iv_i) = -\infty$.
\end {enumerate}
A tree $T$ along with the labels $\ell$ which obey the above rules is
called a \emph{mobile} and will typically be denoted by $\theta =
(T,\ell)$.  If the root
has label $k\in\mathbb{Z}$, the set of such mobiles with $n$ edges will be denoted
by $\Theta^{(k)}_n$, the set of finite mobiles by $\Theta^{(k)}_\mathrm{f}$,
the set of infinite mobiles obtained by labeling trees in
$\Gamma_\infty^\odot$ by $\Theta^{(k)}_\infty$ and finally
$\Theta^{(k)} := \Theta^{(k)}_\mathrm{f}\cup \Theta^{(k)}_\infty$.
As explained in the next subsection, mobiles are an essential tool
in the study of planar maps.

For a finite tree $T$ there is a useful alternative way of describing 
rule (\ref{e:labelinc}) for the labels on $T$, 
see e.g.~\cite{legall:2011}. 
For this purpose we introduce, for each $r\geq 1$, the set
\begin {equation}\label{er_eq}
 E_r = \Big\{(x_1,x_2,\ldots,x_r)\in\{-1,0,1,2,\ldots\}^r :
 \sum_{i=1}^r x_i = 0\Big\}.
\end {equation}
Let $u$ be a black vertex in $T$ of degree $r$, and denote its white parent
by $u^{(0)}$ and its white children by
$u^{(1)},u^{(2)},\ldots,u^{(r-1)}$, ordered from left to right. 
Assign to $u$ an element $(x_1(u),\ldots,x_r(u))$ from $E_r$. 
Having done this for all black vertices $u$, label the white
vertices of $T$ recursively as follows.   First label the root by some fixed
$k$. If for a black vertex $u$ we have that $\ell(u^{(0)}) = y_0$
then let
\begin {equation} \label{labels}
\ell(u^{(j)}) = y_0+\sum_{i=1}^j x_i(u), \qquad 1\leq j \leq r-1.  
\end {equation} 
The elements from $E_r$ thus provide the
increments of the labels of the white vertices clockwise around each
black vertex.   Note that the minimum allowed increment is $-1$, in
accordance with rule (\ref{e:labelinc}). 

The finite sequence $(\ell(u^{(j)}))_{0\leq j \leq r}$ is called a 
\emph{discrete bridge} of length $r$. From this description it is
easy to count the number $\lambda(T)$ 
of different allowed labellings of the
finite tree $T$.  By a standard `balls-and-boxes' argument,
the number of elements in $E_r$ is
$\binom{2r-1}{r-1}$. Therefore, the number of ways of labeling $T$,
given that its root has a fixed label, is
\begin {equation} \label{lambda}
 \lambda(T) = \prod_{u\in V^\bullet(T)} \binom{2 \deg(u)-1}{\deg(u)-1}.
\end {equation}

We conclude this subsection by defining a metric also on the set
$\Theta^{(0)}$. For a mobile $\theta = (T,\ell)$, let $\theta^{[R]}$
be the labeled tree consisting of $T^{[R]}$ and the labels $\ell$
restricted to the white vertices in $T^{[R]}$. Note that
$\theta^{[R]}$ is in general not a mobile since the labels do not
necessarily satisfy the rules listed above. We define a metric
$d_\Theta$ on $\Theta^{(0)}$ by
\begin {equation}\label{theta_metric_eq}
 d_\Theta(\theta_1,\theta_2) =
 \left(1+\sup\left\{R~:~\theta_1^{[R]}=\theta_2^{[R]}\right\}\right)^{-1},\qquad
 \theta_1,\theta_2\in\Theta^{(0)}
\end {equation}
and we equip $\Theta^{(0)}$ with the Borel $\sigma$-algebra.

\subsection{The Bouttier--Di Francesco--Guitter bijection} \label{ss:bdg}
We will recall the rooted and pointed version of the Bouttier-Di
Francesco-Guitter (BDG) bijection 
between mobiles and planar maps
\cite{bouttier:2004}. Consider a finite mobile $\theta = (T,\ell) \in
\Theta_n^{(0)}$ and embed $T$ in the plane. Let
$(c^\circ_i)_{i\in\mathbb{Z}}$ be its white contour sequence and for
each $i$ define the successor of $i$ as
\begin {equation} \label{eq:succ}
 \sigma(i) = \inf\{j>i~:~\ell(c^{\circ}_j) = \ell(c^{\circ}_i)-1\}
\end {equation}
with the convention that $\inf\{\emptyset\}=\infty$. Add a point
$\rho$ to the complement of $T$ in the plane and define
$c^{\circ}_\infty = \rho$. Define the successor of a white vertex
$c^{\circ}_i$ as
\begin {equation} \label{succv}
 \sigma(c^{\circ}_i) = c^{\circ}_{\sigma(i)}.
\end {equation}
Note that every white vertex in the mobile has a unique successor. A
planar map $\map\in \Ma_n$ is constructed from $\theta$, along with a
variable $\epsilon\in\{-1,1\}$, as follows: Draw an arc from each
corner of a white vertex in $\theta$ to its 
successor (in such a way that no arcs
cross). Then delete all black vertices and edges belonging to
$\theta$. The white vertices of $\theta$ along with the external point
$\rho$ are the vertices of $\map$ and $\rho$ takes the role of the
marked vertex. 
The arcs between white vertices take the role of the edges of $\map$.
The root edge of $\map$ is defined as the arc from
$c^{\circ}_0$ to $\sigma(c^{\circ}_0)$ and its direction is determined
by the value of $\epsilon$. If $\epsilon = 1$ ($\epsilon = -1$) it is
directed towards (away from) $c^{\circ}_0$. 

The faces in $\map$
correspond to the black vertices in $\theta$, the degree of a face
being twice the degree of the corresponding black vertex. Furthermore,
the labels of the vertices in $\map$ inherited from the labels in
$\theta$ carry information on distances to the marked point $\rho$.
Namely, if $d_\text{gr}$ is the graph distance on $\map$ and 
$v\neq\rho$ 
is a vertex in $\map$ then
\begin {equation}
 d_\text{gr}(v,\rho) = \ell(v)-\min\{\ell(u)~:~u\in V(\map)\}+1.
\end {equation}

The above construction defines a mapping $\Phi:
\Theta^{(0)}_\mathrm{f}\times\{-1,1\} \rightarrow \Ma_\mathrm{f}$ which is a
bijection. For the inverse construction of $\Phi$, see \cite{bouttier:2004}. 
The mapping $\Phi$
can be extended to infinite elements in $\Theta^{(0)}$ by a
similar description, as follows. If $\theta = (T,\ell)$ is an infinite
mobile we embed $T$ in the plane such that its vertices are
isolated points, as described in Section~\ref{ss:pt}.
Recall that if $T$ has a spine then
$\inf_{i\geq 0}\ell(S^\circ_i)=-\infty$ and if it has a vertex of
infinite degree then $\inf_{i\geq 0}\ell(\iv_i) =
-\infty$. Therefore, every white vertex in the mobile still has a
unique successor which is also a white vertex in the mobile. 
(The other condition, $\inf_{i< 0}\ell(\iv_i) =-\infty$,
ensures that the resulting embedded graph is locally finite.)
The construction of the arcs and the root edge is the same as before and
due to the fact that every successor is contained in the mobile, no
external vertex $\rho$ is needed. The resulting embedded graph, which
we call $\Phi(\theta,\epsilon)$, is thus rooted but not pointed.

In the following proposition we give 
$\Theta^{(0)}$ the topology of~\eqref{theta_metric_eq}, and
$\{-1,1\}$ the discrete topology.
The set $\Theta^{(0)}\times \{-1,1\}$ is given the product
topology.

\begin {proposition}  \label{p:contPhi}\hspace{1cm}
\begin{enumerate}
\item If $\theta\in\Theta^{(0)}_\infty$ then 
$\Phi(\theta,\epsilon)\in\Ma_\infty$.  Thus $\Phi$
extends to a function 
$\Theta^{(0)}\times\{-1,1\}\rightarrow\Ma$.
\item If $\theta\in\Theta^{(0)}_\infty$ then 
$\Phi(\theta,\epsilon)$ is non-pointed and 
one-ended.  It has a unique face of infinite degree
if and only if $\theta$ has a vertex of infinite degree.
\item The function
$\Phi:\Theta^{(0)}\times\{-1,1\}\rightarrow\Ma$
is continuous.
\end{enumerate}
\end {proposition}
\begin {proof}
Let $\theta = (T,\ell)$ be a mobile in $\Theta_\infty^{(0)}$. In the
case when $T$ has a unique infinite spine the proof is nearly
identical to that of~\cite[Proposition~2]{curien:2012a}, which deals
with quadrangulations and labelled trees.  The difference in the case
when $T$ has a unique vertex of infinite degree is first of all that
the left and right contour sequences are independent. Here we need to
use the condition that $\inf_{i\geq 0}\ell(\iv_i) =
\inf_{i<0}\ell(\iv_i) = -\infty$ cf.~Section \ref{ss:mobiles}. Using
this the proof of (1) and (3) proceeds in the same way as in
\cite{curien:2012a}. Secondly, when $T$ has a unique vertex of
infinite degree it is not one--ended in the usual sense. However, for
each $R\geq 0$ the complement of the truncated tree $T^{[R]}$ has
exactly one infinite connected component and this property along with
how the edges in the corresponding map are constructed from $\theta$,
guarantees that $\Phi(\theta,\epsilon)$ is one-ended. We leave the
details to the reader.
\end {proof}

\subsection{Random mobiles}
In this subsection we define a sequence $(\tilde{\mu}_n)_{n\geq 1}$ of
probability measures on $\Theta^{(0)}\times\{-1,1\}$ which
corresponds, via $\Phi$, to the sequence $(\mu_n)_{n\geq 1}$ on $\Ma$.
We start by defining a sequence of probability measures
$(\tilde{\nu}_n)_{n\geq 1}$ on the set of trees $\Gamma^\odot$ which
we then relate to $(\tilde{\mu}_n)_{n \geq 1}$.

Let $(w_i)_{i\geq 0}$ be as in~\eqref{wqdef} 
and assign to each finite tree $T\in\Gamma_\mathrm{f}$ a weight
\begin {equation}
\tilde{W}(T)= \prod_{v\in V^\bullet(T)}w_{\deg(v)}
\end {equation}
and define
\begin {equation} \label{tildenu}
 \tilde{\nu}_n(T) =
 \frac{\tilde{W}(T)}{\sum_{T\in\Gamma_n}\tilde{W}(T)},
\mbox{ if } T\in \Gamma_n\mbox{ (0 otherwise).} 
\end {equation}
Recall that $\lambda(T)$, defined in (\ref{lambda}), denotes the
number of ways one can assign labels to the white vertices of a finite
tree $T$. For each $((T,\ell),\epsilon)\in \Theta^{(0)}\times\{-1,1\}$
and each $n\geq 1$, define
\begin {equation} \label{tildemu}
 \tilde{\mu}_n ((T,\ell),\epsilon) = \tilde{\nu}_n(T) / (2\lambda(T)).
\end {equation}
The following result is then 
well-known~\cite{marckert:2007}.

\begin {lemma}\label{lem:mu}
 For each $n\geq 1$, the measure $\mu_n$ is the image of
 $\tilde{\mu}_n$ by the mapping $\Phi$.
\end {lemma}

Note from (\ref{tildemu}) that a random element $((T,\ell),\epsilon)
\in\Theta^{(0)}\times\{-1,1\}$ distributed by the measure
$\tilde{\mu}_n$ can be constructed by:
\begin {enumerate}
 \item Selecting a tree $T$ according to the measure $\tilde{\nu}_n$.
\item Given the tree $T$, labeling its root by $0$ and 
\begin {enumerate}
\item [(a)] assigning a labeling $\ell$ to the white 
vertices of $T$ uniformly from the set of allowed labelings;  
or equivalently
\item [(b)] for every $r\geq 1$ assigning independent uniform 
elements from $E_r$ to each black vertex of degree $r$ and
defining $\ell$ recursively as described in and above (\ref{labels}).
\end {enumerate}
\item Selecting independently an element $\epsilon$ uniformly from $\{-1,1\}$.
\end {enumerate}
Note that the only `part' of the measure $(\tilde{\mu}_n)_{n\geq 1}$
which depends on the parameters  $(w_i)_{i\geq 0}$ of the model is the
`tree part' $(\tilde{\nu}_n)_{n\geq 1}$. We will therefore first focus our
attention on the latter.

Also note, for future reference, that step (2b) above has the following
alternative description.  Let $X_1,X_2,\dotsc$ be independent,
all with the same distribution given by
\begin {equation}\label{bridgelaw}
 \mathbb{P}(X_j=k) = 2^{-k-2}, \quad k = -1,0,1,2,\ldots.
\end {equation}
A uniformly chosen element of $E_r$ has the same distribution
as the sequence $(X_1,X_2,\dotsc,X_r)$ \emph{conditioned on}
$\sum_{j=1}^r X_j=0$.

\section{The local limit}\label{s:conv}

This section is devoted to the proof of Theorem~\ref{thm:locallimit}.
We start by describing the weak limit of the \emph{unlabelled}
mobiles, that is the sequence $(\tilde\nu_n)_{n\geq1}$, 
see Theorem~\ref{th:tildenu}.
We then describe in Section~\ref{ss:labelledmobiles} how to `put the labels back on',
and this gives us a proof of Theorem~\ref{thm:locallimit}.  
The proof of Theorem~\ref{th:tildenu} in turn relies
on the theory of simply generated trees, which is described in
Section~\ref{s:sgt}.  The proof of Theorem~\ref{th:tildenu}
is given in Section~\ref{ss:tildenuproof}.

\subsection{Weak convergence of unlabelled mobiles} \label{ss:tildenu}

In this subsection we state a convergence theorem for the measures
$\tilde{\nu}_n$ which we prove in  Section~\ref{ss:tildenuproof}. 
Recall that vertices in odd generations are coloured
black and even generations white. 
Recall also the definition of $(\pi_i)_{i\geq 0}$ 
and $\xi$ from (\ref{pi}).

Let $\xi^\circ$ and $\xi^\bullet$ be random 
variables in $\{0,1,2\ldots\}$ with distributions given by
\begin {equation}
  \P(\xi^\circ= i) = \pi_0 (1-\pi_0)^i, \quad i\geq 0
\end {equation}
and (when $\pi_0 < 1$)
\begin {equation}
  \P(\xi^\bullet= i) = \pi_{i+1}/(1-\pi_0), \quad i\geq 0.
\end {equation}
(These appear in~\cite[Proposition~7]{marckert:2007}, 
where the law of $\xi^\circ$ is denoted $\mu_0$ 
and the law of $\xi^\bullet$ is denoted $\mu_1$.)
Also, let $\hat{\xi}^\circ$ and $\hat{\xi}^\bullet$ be random 
variables in $\{1,2,\ldots\}\cup \{\infty\}$ having distributions
given by
\begin{equation}
 \P(\hat{\xi}^\circ= i) = \pi_0^2 i (1-\pi_0)^{i-1}, \quad i \geq 1
\end{equation}
and 
\begin {equation}
  \P(\hat{\xi}^\bullet = i) = \left\{
\begin {array}{cc}
i\pi_{i+1}/\pi_0 & \text{if}~1\leq i<\infty \\
(1-\kappa)/\pi_0 & \text{if}~i=\infty.
\end {array}\right.
\end {equation}
Thus $\hat{\xi}^\circ$ is the sized-biased version
of ${\xi}^\circ$, and similarly for
$\hat{\xi}^\bullet$ in the case $\kappa=1$.
We now define a probability measure $\tilde{\nu}$ 
on infinite trees, by describing a random tree 
$\tilde{\T}$ with law
$\tilde{\nu}$.  
We let $\tilde{\T}$ be a modified multi-type Galton--Watson
tree having four types of vertices: 
\emph{normal} black and white vertices
and \emph{special} black and white vertices. 
The root is white and is declared to be
a special white vertex.  Vertices have offspring independently
according to the following description. Special white vertices give
birth to black vertices, their number having the law of 
$\hat{\xi}^\circ$;  one of the black children is chosen 
uniformly at random to be
special and the rest are declared normal. 
Special black vertices give birth to
white vertices, their number having 
the law of  $\hat{\xi}^\bullet$. If the
number of white children is finite, one of them is chosen uniformly to
be declared special and the rest  normal. 
If the number of white children is
infinite, all of them are declared to be normal.  Normal white
vertices give birth to normal black vertices, 
their number having the law of
$\xi^\circ$, and normal black vertices give birth to normal white
vertices, their number having the law of~$\xi^\bullet$.

We will now describe how a typical tree $\tilde{\T}$ looks like
depending on the parameters $(w_i)_i$. Define 
\begin {equation}
\tilde{\kappa} = (\kappa+\pi_0-1)/\pi_0
\end {equation}
and note that $\tilde{\kappa} \leq 1$, and
that $\tilde{\kappa} < 1$ if and only if $\kappa < 1$. First of all,
the special vertices define a spine which is 
infinite if and only if $\tilde{\kappa} = 1$.
If $\tilde{\kappa} < 1$ the
spine ends with a black vertex of infinite degree, which has only
normal white children. In that case its
length $\tilde{L}$ (number of edges) has a 
geometric distribution:
\begin {equation} \label{mobilelength}
 \mathbb{P}(\tilde{L} = 2n+1) = (1-\tilde{\kappa}) \tilde{\kappa}^n,
 \quad n\geq 0.
\end {equation}
The normal children of the vertices on the spine are root vertices of
independent two-type Galton--Walton processes where white (resp.~black)
vertices have offspring distributed as $\xi^\circ$
(resp.~$\xi^\bullet$). We will call these Galton--Watson processes
\emph{outgrowths} from the spine. 
If $\pi_0 < 1$ the mean number of offspring in two
consecutive generations in an outgrowth is given by
\begin {equation}
 \mathbb{E}(\xi^\circ)\mathbb{E}(\xi^\bullet) = \frac{1-\pi_0}{\pi_0}
 \frac{\kappa-1 + \pi_0}{1-\pi_0} = \tilde{\kappa}.
\end {equation}
Thus the outgrowths are critical if $\tilde{\kappa} = 1$ and
sub-critical otherwise. In both cases they are almost surely finite
and therefore the tree $\tilde{\T}$ is at most one ended.

To summarize, we have the two following qualitatively different
cases. In the case $\tilde{\kappa}=1$ the tree $\tilde{\T}$ has an
infinite spine consisting of special white and black vertices. The
outgrowths from the spine are independent critical two-type
Galton--Watson processes as described above. In the case
$\tilde{\kappa}<1$ the tree $\tilde{\T}$ has a finite spine with 
geometrically distributed length~\eqref{mobilelength}. The spine
consists of special black and white vertices and has outgrowths, 
which are independent, sub-critical two-type Galton--Watson 
processes. In the extreme case $\tilde{\kappa}=0$ we have 
$\pi_0 = 1$ and thus $\kappa =0$. 
In this case $\tilde{\T}$ is deterministic and consists of 
a white root having a single black vertex of infinite degree 
and all outgrowths empty.

We have the following.

\begin {theorem} \label{th:tildenu}
 The sequence of measures $(\tilde{\nu}_n)_{n\geq 1}$ on
 $\Gamma^\odot$ converges weakly to $\tilde{\nu}$ (the law of
 $\tilde{\T}$) as $n\rightarrow\infty$ in the topology generated by
 $d_\Gamma$.
\end {theorem}

The proof uses the theory of simply generated trees and is therefore
deferred until Section~\ref{ss:tildenuproof}.

\subsection{Weak convergence of labelled mobiles}
\label{ss:labelledmobiles}
We will now use Theorem~\ref{th:tildenu} 
to construct an infinite random mobile $\vartheta$ in
$\Theta^{(0)}$
and show that it appears as the
limit of the sequence $(\theta_n)_{n\geq 1}$ distributed by
$(\tilde{\mu}_n)_{n\geq 1}$.
Recall that $\tilde\mu_n$ is obtained from
$\tilde\nu_n$ by `putting on the labels'
and also sampling the direction $\epsilon$
of the root edge.

To construct $\vartheta$ start with the random tree
$\tilde{\T}$ with law $\tilde\nu$.
Given $\tilde{\T}$, assign independently to each of its
black vertices $v$ of finite degree $r$ an element
$B(v)$ selected uniformly from
$E_r$. If $\tilde{\T}$ has a black vertex $s$ 
of infinite degree, assign to
that vertex a sequence of independent 
random variables $(X_i)_{i\in \mathbb{Z}}$
which are independent of the $B(v)$ and with the 
law~\eqref{bridgelaw}.  Define the labels $\ell(v)$
by first labelling the root $\ell(\varnothing)=0$, and then
letting the $B(v)$ determine the increments around $v$
as described above~\eqref{labels}, and in addition letting
the increments around $s$ be given by the sequence
$(X_i)_{i\in \mathbb{Z}}$.

Let $\epsilon\in\{-1,+1\}$ be uniformly chosen and
independent of the random variables in the paragraph above.
Finally, let $\vartheta=(\tilde{\T},\ell)$ be the corresponding
infinite mobile.

\begin{lemma}\label{lem:tildemu}
Writing $\tilde\mu$ for the law of the 
pair $(\vartheta,\epsilon)$ we have that
$\tilde\mu_n\Rightarrow\tilde\mu$.
\end{lemma}
\begin{proof}
Let $\theta_n=((T_n,\ell_n),\epsilon_n)$ have law
$\tilde\mu_n$.  
Since $T_n\Rightarrow\tilde\T$ it suffices
to show that $\ell_n\Rightarrow\ell$ where $\ell$
is the labeling of $\vartheta$ above.  In both
$\theta_n$ and $\vartheta$, the label increments around
different black vertices are independent.  The increments
around a black vertex of finite degree are in both
cases uniformly chosen from the set $E_r$ in~\eqref{er_eq},
and we are thus done if we show that the increments around
a vertex of degree $\omega(n)\to\infty$ converge to
the corresponding sequence $(X_i)_{i\in\ZZ}$.  This
follows from the following claim, which is
easily verified by explicit `balls-in-boxes' enumeration
and Stirling's approximation.

\emph{Claim:}  Let $X_1,X_2,\dotsc$ be independent
and with law given in~\eqref{bridgelaw}.  Then for
each fixed $R\geq1$ and all
$a_1,\dotsc,a_R\in\{-1,0,1,\dotsc\}$ we have that
\[
\lim_{n\rightarrow\infty}
P\Big(X_1=a_1,\dotsc,X_R=a_R\;\Big|\sum_{j=1}^nX_j=0\Big)
=P(X_1=a_1,\dotsc,X_R=a_R).
\]
\end{proof}

Now we can prove the weak convergence of the probability
measures $\mu_n$ on planar maps:
\begin{proof}[Proof of Theorem~\ref{thm:locallimit}]
By Lemma~\ref{lem:mu} we have that $\mu_n=\Phi(\tilde\mu_n)$,
and by Proposition~\ref{p:contPhi} that $\Phi$
is continuous.  The weak convergence of $\mu_n$ towards $\mu$ follows from
Lemma~\ref{lem:tildemu}. 
The limit is one--ended by Proposition~\ref{p:contPhi} and the 
presence of a face of infinite degree when $\kappa < 1$ follows from 
the existence of a black vertex of infinite degree in 
$\tilde{\T}$. 
When $\kappa = 0$ the tree
$\tilde{\T}$ is deterministic and consists of a single black vertex of
infinite degree with white neighbours of degree 1,
and can be seen as the local limit
as $r\to\infty$ of a single black vertex of
degree $r$ with white neighbours of degree 1.  
The labels are determined by a uniformly chosen
element of $E_r$, and it follows that the corresponding
map is a uniformly chosen plane tree with $r+1$ vertices.
\end{proof}

\subsection{Simply generated trees} \label{s:sgt}
In this section we describe the model of 
\emph{simply generated trees} and
state a general convergence theorem by 
Janson~\cite{janson:2012sgt}.  In the following section we then describe a bijection
$\Psi:\Gamma_\mathrm{f}\rightarrow \Gamma_\mathrm{f}$ 
which relates the probability
measures $(\tilde{\nu}_n)_{n\geq 1}$ (defined in (\ref{tildenu})) to the
simply generated trees. We then extend $\Psi$ to a mapping
$\Psi:\Gamma\rightarrow\Gamma^\odot$ and show that it is 
continuous.
This will allow us to use the convergence results for the
simply generated trees to prove Theorem~\ref{th:tildenu}.

Simply generated trees are random trees defined by a sequence of
probability measure $(\nu_n)_{n\geq 1}$ on $\Gamma$ as follows. Let
$(w_i)_{i\geq 0}$ be a sequence of non-negative numbers and assign to
each finite tree $T$ a weight
\begin {equation}
W(T) = \prod_{v\in V(T)}w_{\deg(v)-1}
\end {equation}
and define
\begin {equation}\label{nu_def_eq}
 \nu_n(T) = \frac{W(T)}{\sum_{T'\in\Gamma_n} W(T')},
\mbox{ if } T\in \Gamma_n \mbox{ (otherwise 0)}.
\end {equation}
We assume that the weight
sequence $(w_i)_{i\geq 0}$ is defined as in and above
(\ref{wqdef}). Janson obtained a general convergence theorem for
simply generated trees in the local topology which applies for every
choice of weight sequence \cite{janson:2012sgt}. Before stating the
theorem we need a few definitions.

Let $\pi_i$ be defined as in (\ref{pi}) and as before let $\xi$ be a
random variable distributed by $(\pi_i)_{i\geq 0}$ with 
mean $\kappa\in[0,1]$.
In the extreme case $\kappa = 0$ one has 
simply $\pi_i = \delta_{i,0}$. 
Define a random variable $\hat{\xi}$ on $\{0,1,\ldots\}\cup\{\infty\}$ by
\begin {equation}\label{hat_xi_eq}
 \P(\hat{\xi} = k) = \left\{
\begin {array}{ll}
k\pi_k & \text{if}~k<\infty \\
1-\kappa & \text{if}~k=\infty.
\end {array}\right.
\end {equation}
We will now construct a modified Galton--Watson tree $\sgt$ which
arises as the local limit of the simply generated trees. We will
denote the law of $\sgt$ by $\nu$. The tree was originally defined
by Kesten~\cite{kesten:1986} (for $\kappa = 1$) and Jonsson and
Stefánsson~\cite{jonsson:2011} (for $\kappa < 1$) but here we follow
Janson's construction~\cite{janson:2012sgt}. 

In $\sgt$  there will be two types of
vertices, called normal vertices and special
vertices. The root is declared to be special. Normal vertices have
offspring independently according to the distribution $\xi$ whereas
special vertices have offspring independently according to the
distribution $\hat{\xi}$. All children of normal vertices are normal
and if a special vertex has infinite number of children they are all
normal. (In our case we assume that the infinite number of children is
ordered from left to right as explained in the beginning of Section
\ref{s:trees} whereas conventionally they are ordered as
$\mathbb{N}$. This small difference will clearly not affect the main
result). Otherwise, all children of a special vertex are normal except
for one which is chosen uniformly to be special.

The tree $\sgt$ has different characteristics depending on whether
$\xi$ is critical ($\kappa = 1$) or sub--critical
($\kappa < 1$). In the critical case $\sgt$ has a
unique infinite spine composed of the special nodes and the outgrowths
from the normal children of the vertices on the spine are independent
critical Galton--Watson trees distributed by $\xi$.  In the
sub--critical case the linear graph composed of the special nodes is
almost surely finite ending with a special node having infinite number
of normal children. It is thus a finite spine and it has a length $L$ 
distributed by $\P(L = i) = (1-\kappa)\kappa^i$
for $i\geq0$. The outgrowths from
the normal children of the vertices on the spine are then
sub--critical Galton--Watson trees distributed by $\xi$. In the
extreme case $\kappa= 0$, $\sgt$ has a spine of length 0 and the
root has an infinite number of normal children which have no children
themselves. In this case the tree is therefore deterministic.

Since the UIPTree appears repeatedly in this paper it is useful to note that $\sgt \sim$ UIPTree when $w_i = 1$ for all $i$ in which case $\kappa = 1$. In this case $\pi_i = 2^{-i-1}$. 

\begin {theorem}[Janson \cite{janson:2012sgt}] \label{th:janson}
For any sequence $(w_i)_{i\geq 0}$ such that $w_0>0$ and $w_k >
0$ for some $k\geq 2$ the sequence of measures $(\nu_n)_{n\geq 1}$ on
$\Gamma$ converges weakly towards $\nu$ (the distribution of
$\sgt$) with respect to the topology generated by $d_\Gamma$.
\end {theorem}

The case when $\kappa = 1$ was first proved implicitly by Kennedy \cite{kennedy:1975} and later by Aldous and Pitman \cite{aldous:1986}.  The special case when $0<\kappa < 1$ and $w_i \sim c i^{-\beta}$, $c >0$, $\beta > 2$ was originally proved by Jonsson and Stefánsson \cite{jonsson:2011}. Janson, Jonsson and Stefánsson also proved a special case when $\kappa = 0$ \cite{janson:2011}.

\subsection{The bijection $\Psi$} \label{ss:bijection}

The mapping $\Psi:\Gamma_\mathrm{f}\rightarrow\Gamma_\mathrm{f}$
which we describe now will map the model of 
simply generated trees onto the unlabelled mobiles.
In order to describe it we will temporarily violate our
colouring convention of Section~\ref{s:trees}.
Instead of coulouring  even generations
white and odd generations black, we will now
colour vertices of degree one (that is, leaves) white
and all other vertices black.  
The mapping $\Psi$ will then
precisely map the white vertices to even generations
and black vertices to odd generations.

Start with a finite tree $T\in\Gamma_n$ having $n\geq 1$ edges and
colour the vertices as described above. 
Let $v$ be a white vertex (leaf) and note that $v$ appears exactly once
in the contour sequence $(c_i)_{i\in\ZZ}$ (up to
periodicity).  Thus $v=c_{j(v)}$ for some $j(v)$.
Define
\begin {equation}
 \eta(v) = \max\{k~:~c_{j(v)}\succ c_{j(v)+1}
\succ\dotsb\succ c_{j(v)+k-1}\succ c_{j(v)+k}\}
 \end {equation}
and define the sequence $(b_i(v))_{1 \leq i \leq \eta(v)}$ by 
$b_i(v) = c_{j(v)+\eta(v)-i+1}$.  
In words, $\eta(v)$ is given by following the contour sequence
(clockwise) from $v$ for as long as this coincides
with the ancestry line of $v$.  Then $b_1(v)$ is the
earliest ancestor of $v$ on this part of the contour sequence
and $(b_i)_{1\leq i\leq \eta(v)}$ traces the ancestry line
from $b_1(v)$ to the parent $b_{\eta(v)}(v)$ of $v$.
By definition all the vertices in
$(b_i(v))_{1\leq i \leq \eta(v)}$ are black and it is straightforward
to check that
\begin {equation}
 \bigcup_{v~:~\mathrm{deg}(v)=1}
 \{v,b_1(v),b_2(v),\ldots,b_{\eta(v)}(v)\} = V(T).
\end {equation}
Now, construct a new tree $T'$ from $T$ by drawing an arc from 
each white
vertex $v$ to the corresponding black vertices in
$(b_i(v))_{1\leq i\leq \eta(v)}$.  Then throw away the edges from
$T$ and let the arcs just drawn become the
edges of $T'$. The root of $T'$ is defined as the
\emph{first} white vertex in the \emph{right} 
contour sequence. The left to right ordering of the children of
a white vertex in $T'$ is inherited from the ordering of
$(b_i(v))_{1\leq i \leq \eta(v)}$.
See Fig.~\ref{f:simpmob} for an example.

We let the vertices of $T'$ inherit the colours of the corresponding
vertices in $T$ and note that $T'$ then has a white root and that
every even generation is white and every odd generation is
black, as claimed previously. 
The black vertices in $T'$ have degree equal to
their original outdegree, i.e.~their degree is reduced by one. 
(This is true also for the black root if one attaches
a half-edge to it, as represented in Fig.~\ref{f:simpmob}.)
The degree of the
white vertex $v'$ in $T'$ corresponding to the vertex $v$ in $T$ is
\begin {equation} \label{degv}
 \deg(v') = \eta(v).
 \end {equation}
This construction defines a bijection $\Psi$ from
$\Gamma_\mathrm{f}$ to itself. For the inverse construction see
\cite{janson:2012sl}. 

Define $\tilde\Gamma$ to be the set of trees in $\Gamma$
whose right contour sequence visits
infinitely many vertices of degree 1 (that is, leaves).
We consider the latter condition to be satisfied
for a finite tree by periodicity of the contour
sequence.  It is straightforward to see that the measures 
$(\nu_n)_{n\geq1}$ and $\nu$ are all supported on
$\tilde\Gamma$.
The function $\Psi$ can be extended to a function
$\Psi: \tilde\Gamma \rightarrow \Gamma^\odot$ by exactly the same
construction as in the finite case.  

\begin{figure} [t]
\centerline{\scalebox{0.9}{\includegraphics{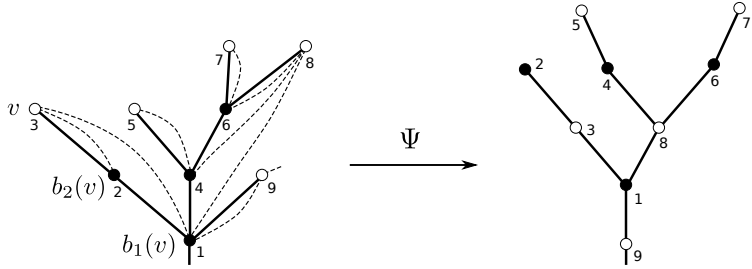}}}
\caption{An example of the bijection $\Psi$. The original tree (containing 9 vertices numbered 1-9) is
  drawn on the left hand side in solid lines. The tree obtained by $\Psi$ is drawn on top of it in in dashed lines and on the right hand side in solid lines. The roots (vertices 1 and 9 on the left and right respectively) are indicated by half-edges.} \label{f:simpmob}
\end{figure}

\begin {proposition}
 The extended mapping $\Psi: \tilde\Gamma \rightarrow \Gamma^\odot$ is 
 continuous. 
\end {proposition}
\begin {proof}
Let $T, T_1,T_2,\dotsc\in\tilde\Gamma$ with
$T_n\rightarrow T$ in the local topology, i.e.~for each $R\geq 0$ there is an $n_0$ such that
\begin {equation} \label{convT}
 T_n^{[R]} = T^{[R]}
\end {equation}
for all $n\geq n_0$. If $T$ is finite it follows immediately that
$\Psi(T_n)\to\Psi(T)$, hence we assume that $T$ is infinite. First
look at the case when $T$ has an infinite spine $S = (S_k)_{k\geq
  0}$. Denote by $(S_{k_i})_{i\geq 0}$ the subsequence of vertices on
the spine which have the property that their rightmost child is not on
the spine (the outgrowth to the right of it is nonempty). Furthermore,
for each $S_{k_i}$, let $v_{i}$ be the first white vertex following it
in the right contour sequence (it is necessarily in the nonempty
outgrowth to the right of $S_{k_i}$).  The sequences $(S_{k_i})_{i\geq
  0}$ and $(v_i)_{i\geq 0}$ are infinite due to the definition of
$\tilde \Gamma$ and $v_0,S_{k_0},v_1,S_{k_1},\ldots$ is the infinite
spine in $\Psi(T)$.

To prove the continuity of $\Psi$ we need to show that for any fixed
$R \geq 0$ the sequence $\Psi(T_n)^{[R]}$ is eventually constant. We
choose an $R'$ large enough such that $T^{[R']}$ contains 
$v_{\lceil R/2 \rceil}$ and the vertices 
$S_{0},S_1,\ldots,S_{k_{\lceil R/2  \rceil}}$ on the spine along 
with their (finite) outgrowths.
Then any vertex of $T$ not in $T^{[R']}$ maps outside
$\Psi(T)^{[R]}$, and thus $\Psi(T)^{[R]}\subseteq\Psi(T^{[R']})$.
Similarly, if $n$ is large enough that 
$T_n^{[R']}=T^{[R']}$ then $\Psi(T_n)^{[R]}\subseteq\Psi(T_n^{[R']})$.
Thus for such $n$ we have
\begin {equation}
 \Psi(T_n)^{[R]} = \Psi(T_n^{[R']})^{[R]} = \Psi(T^{[R']})^{[R]} =
 \Psi(T)^{[R]}.
\end {equation}
When $T$ has a vertex of infinite degree the proof goes along the same
lines and is left to the reader.
\end {proof}
The next result is originally from \cite{janson:2012sl} and 
the proof follows directly from the construction of the 
bijection $\Psi$ on the set of finite trees.
\begin {lemma} [\cite{janson:2012sl}] \label{l:imagepsi}
Let $(w_i)_{i\geq 0}$ be defined as in and above (\ref{wqdef}) and let
$(\tilde{\nu}_n)_{n\geq 1}$ be the sequence of measures defined in
(\ref{tildenu}). Let $(\nu_n)_{n\geq 1}$ be as in~\eqref{nu_def_eq}.
Then for each $n\geq 1$, $\tilde{\nu}_n$ is the
image of $\nu_n$ by the mapping $\Psi$.
\end {lemma}

\subsection{Proof of Theorem \ref{th:tildenu}}
\label{ss:tildenuproof}
 Let $(T_n)_{n\geq 1}$ be a sequence of trees distributed by
 $(\tilde{\nu}_n)_{n\geq 1}$. By Theorem \ref{th:janson} and Lemma
 \ref{l:imagepsi} it holds that $T_n \rightarrow \Psi(\sgt)$ in
 distribution. The only thing left is to show that $\Psi(\sgt) =
 \tilde{\T}$ in distribution. 
In this proof we follow the colouring convention of the previous
subsection, that is vertices of degree one in $\sgt$ are white
and the rest black.

Firstly, it is straightforward to see that
$\Psi(\sgt)$ has a (unique) infinite spine 
if and only if $\sgt$ has an infinite
spine, and that 
$\Psi(\sgt)$ has a (unique) vertex of infinite degree 
if and only if $\sgt$ has a vertex of  infinite degree.
Indeed, if $\sgt$ has a vertex of infinite degree then
(the image of) this vertex has infinite degree also
in $\Psi(\sgt)$.  If $\sgt$ has an infinite
spine $S_0,S_1,\dotsc$ then the infinite spine of
$\Psi(\sgt)$, call it $S'$, may be found as follows. The black vertices in $S'$ are the black vertices in $S$ whose rightmost children are not special and their order in $S'$ is inherited from their order in $S$ . A white vertex in $S'$ preceding a given black vertex $v'$ in $S'$ is the first white vertex in the right contour sequence of $\sgt$ which appears after the first occurrence of the black vertex in $\sgt$ corresponding to $v'$.

We start by checking that the black vertices have the correct
outdegree distribution in $\Psi(\sgt)$, then that 
$\Psi(\sgt)$ has the independence structure of 
a (modified) Galton--Watson 
tree, and finally that the white vertices have the right
outdegree distribution.
We will
divide the black vertices in $\sgt$ into three categories:
\begin {enumerate}
 \item Normal black vertices.\label{i:nb}
\item Special black vertices which have a special child as the
  rightmost child. \label{i:sbr}
\item Special black vertices which do not have a special child as the
  rightmost child. \label{i:sbnr}
\end {enumerate}
The vertices belonging to (\ref{i:nb}) and (\ref{i:sbr}) correspond
exactly to the normal black vertices in $\Psi(\sgt)$,
and the vertices in (\ref{i:sbnr})
correspond to the special black vertices. 
Indeed, a vertex of type (\ref{i:nb}) has outdegree 
in $\sgt$ taking value
$i\geq1$ with probability
\begin {equation}
\P(\xi = i ~|~ \xi > 0) = \pi_i/(1-\pi_0)
\end {equation}
and the probability that a vertex of type (\ref{i:sbr})
has outdegree $i\geq1$ in $\sgt$ equals the conditional
probability that a special vertex in $\sgt$ has $i$ children
given that the rightmost child is special, which is
\begin {equation} \label{rightmost}
 \frac{\P(\hat{\xi}=i)/i}{\sum_{j=1}^\infty \P(\hat{\xi}=j)/j } =
 \pi_i/(1-\pi_0).
\end {equation}
Since the mapping $\Psi$ reduces the degree of black vertices by 1 we
see that the outdegree of the black vertices in
$\Psi(\sgt)$ corresponding to (\ref{i:nb}) and (\ref{i:sbr}) 
takes value $i\geq0$ with probability
$\pi_{i+1}/(1-\pi_0)$, in agreement with the
distribution of $\xi^\bullet$. 

The probability that a vertex of type (\ref{i:sbnr})
has outdegree $1\leq i<\infty$ in $\sgt$ equals the
conditional probability that a special vertex in $\sgt$
has $i$ children given that the rightmost child is not
special, which is
\begin {equation} \label{notspecial}
 \frac{\P(\hat{\xi}=i)(1-1/i)}{1-\sum_{j=1}^\infty\mathbb{P}(\hat{\xi}=j)/j
 } = (i-1)\pi_i/\pi_0.
\end {equation}
Similarly, vertices of type (\ref{i:sbnr}) have infinite 
degree with probability
$P(\hat\xi=\infty)/\pi_0=(1-\kappa)/\pi_0$. 
Again, by shifting by one we find
that this agrees with the distribution of $\hat{\xi}^\bullet$.

We now consider the white vertices.
For a white vertex $v$ in $\sgt$ we
recall the definitions of $\eta(v)$ and 
$(b_i(v))_{1\leq i \leq\eta(v)}$ from 
Section~\ref{ss:bijection}. We will suppress the
argument $v$ in the following for easier notation. The white vertex in
$\Psi(\sgt)$ corresponding to $v$ in $\sgt$ will be denoted by
$v'$.  If $v'=\varnothing$ then its offspring
(in $\Psi(\sgt)$) correspond exactly to the black vertices
$b_1,\dotsc,b_\eta$ in $\sgt$, whereas if 
$v'\neq\varnothing$ then its offspring
correspond to $b_2,\dotsc,b_\eta$, with
$b_1$ corresponding to the parent of $v'$.
Conditioning on the number of offspring of a white
vertex $v'\in\Psi(\sgt)$ thus corresponds to
conditioning on the length of a `rightmost' ancestry path in 
$\sgt$.  From this it is easy to see that the
children of $v'$ have independent numbers of offspring
in $\Psi(\sgt)$, and furthermore that the same holds for the 
white vertices forming the following generation
in $\Psi(\sgt)$.  This implies that $\Psi(\sgt)$
has the correct independence structure.

It remains to check that the white vertices have the correct
offspring distribution.  Starting with the root $\varnothing$,
its offspring in $\Psi(\sgt)$, ordered from left to right,
consist of: firstly, some number $i\geq0$
of black vertices of type (\ref{i:sbr}) above; next,
one vertex of type (\ref{i:sbnr}); and finally some 
number $j\geq0$ of vertices of type (\ref{i:nb}).
The number of offspring of $\varnothing$ is then $k=i+j+1$,
and this occurs with probability
\[
(1-\pi_0)^i\pi_0(1-\pi_0)^j\pi_0=\pi_0^2(1-\pi_0)^{k-1}.
\]
Here the first factors $(1-\pi_0)^i\pi_0$ are due to the
occurrence, in the sequence $(b_i)_{1\leq i\leq\eta}$,
of $i$ special vertices each of whose rightmost child
is not special, followed by one special vertex whose
rightmost child is special.  The remaining factors
$(1-\pi_0)^j\pi_0$ are due to the occurrence
of $j$ normal vertices with at least one offspring
each, followed by one with no offspring.  Summing over
the possible values of $i$ gives the probability
$k\pi_0^2(1-\pi_0)^{k-1}$ of $\varnothing$
having $k$ offspring, in agreement with
 the distribution of $\hat\xi^\circ$.
Note that, given the outdegree (number of offspring)
of $\varnothing$, the black child of type (\ref{i:sbnr})
is uniformly distributed.

Having dealt with the root of $\Psi(\sgt)$,
the remaining white vertices $v$ in 
$\sgt$ are divided into two
categories:
\begin {enumerate}
\item either $\eta=1$, or $\eta>1$ and $b_2$ is normal; 
 \label{i:normal}
\item $\eta > 1$ and  $b_{2}$ is
  special. \label{i:special}
\end {enumerate}

White vertices $v$ in category (\ref{i:normal}) correspond exactly to
the normal white vertices in $\Psi(\sgt)$. In this case, each
black vertex in $(b_i)_{2\leq i \leq \eta}$ is normal and has
at least one child in $\sgt$, whereas $v$ has no child in $\sgt$.
Thus, by (\ref{degv}) the outdegree of the white vertex $v'$
in $\Psi(\sgt)$ satisfies
\begin {equation}
 \P(\out(v') = i) = \P(\eta - 1 = i) = 
(1-\pi_0)^i \pi_0 ,\quad i\geq 0,
\end {equation}
agreeing with the distribution of $\xi^\circ$.

Case (\ref{i:special}) is handled in the same way as the case
$v'=\varnothing$, showing that the outdegree in
$\Psi(\sgt)$ is distributed as $\hat{\xi}^\circ$.  
Thus we have shown
that $\Psi(\sgt) = \tilde{\T}$ in distribution.
\qed

\section{Recurrence}\label{s:rec}

In this section we prove Theorem~\ref{thm:recurrence}.
As mentioned previously, we will rely on a general result
established in \cite{gurel:2013}, which we begin by describing.
Suppose $(G_n)_{n\geq1}$ is a sequence of finite graphs,
and that in each graph $G_n$ is singled out a 
\emph{root vertex} $o_n$.    One may define a local limit of such 
a sequence of rooted graphs $(G_n,o_n)$ in much the same way as in
Section~\ref{s:planarmaps}:  $(G_n,o_n)$ converges locally to 
$(G,o)$ if for each $r$, the graph ball of $(G_n,o_n)$ centered
at $o_n$ with radius $r$ eventually equals the corresponding graph
ball of $(G,o)$.  Now suppose that each $(G_n,o_n)$ is a random,
planar graph, viewed up to isomorphism of rooted graphs.
We say that the root $o_n$ has the \emph{stationary distribution}
if, given $G_n$, the probability that $o_n$ is some fixed vertex
$v$ of $G_n$ is proportional to the degree of $v$.
Building on results by Benjamini and Schramm \cite{benjamini:2001}, who considered the 
case when the maximum degree in $G_n$ is uniformly bounded,
Gurel-Gurevich and Nachmias proved the following:
\begin{theorem}[\cite{gurel:2013}]\label{ggn_thm}
Let $(G_n,o_n)$ be a sequence of finite, random planar graphs such 
that $o_n$ has the stationary distribution for each $n$, and
such that $(G_n,o_n)$ converge weakly to $(G,o)$ in the local
topology.  If the degree distribution of $o$ in $G$ has an exponential
tail, then $G$ is almost surely recurrent.
\end{theorem}

In applying this result to our situation, we take the root vertex
$o_n$ to be the origin $e_-$ of the root edge $e$.
There are two main steps to applying Theorem~\ref{ggn_thm}:
firstly, proving that $e_-$ has the
stationary distribution under each $\mu_n$; 
and secondly, proving that the degree of
$e_-$ has an exponential tail under $\mu$.  

The claim that $e_-$ has the stationary distribution is equivalent to
the statement that the directed edge $e$ is chosen uniformly among all
directed edges of $\map$. By a simple calculation, this follows from
the fact that the probability assigned by $\mu_n$ to a rooted map does
not depend on the choice of root edge. To prove
Theorem~\ref{thm:recurrence} it therefore suffices to show that $e_-$
has an exponential tail under $\mu$.

\subsection{Bound on the degrees in $M$}

 Recall from Sections~\ref{ss:tildenu}--\ref{ss:labelledmobiles}
the infinite mobile $\vartheta=((\tilde\T,\ell),\epsilon)$ 
which (via the BDG bijection)
defines the map $M$.  The tree  $\tilde\T$ consists
of a spine of special black and white vertices, with
finite normal trees attached on the left and right sides. 
In this section we will describe a slightly different way
of constructing $\vartheta$ which will let us deduce
bounds on the degrees of the vertices in $M$.  

Recall the random variables $\xi^\bullet$ and $\xi^\circ$
of Section~\ref{ss:tildenu};  they have the distribution
of the outdegree of normal black and white vertices in
$\tilde\T$, respectively. We will assume in this section that $\pi_0
< 1$, or equivalently that $\kappa>0$.  
When $\pi_0 = 1$, $M$ is the UIPTree and its degree 
distribution (and the fact that it is recurrent) 
is well known. It is essential for our argument
that $\xi^\circ$ is geometrically distributed:
$\PP(\xi^\circ=i)=\pi_0(1-\pi_0)^i$ for $i\geq0$.
Also recall that for each normal black vertex $v$ of
$\tilde\T$ with outdegree $r\geq1$, the clockwise label
increments around $v$ form a discrete bridge 
$(X_1^{(r)},\dotsc,X_{r+1}^{(r)})$ with law described 
just below~\eqref{bridgelaw}.  We shall be particularly 
interested in the event that $X_{r+1}^{(r)}\geq1$;
that is to say, the last increment in the clockwise order
is 1 or more, or equivalently the \emph{first} increment
in the anticlockwise order is $-1$ or less.
For reasons that will become clear soon, we define
\[
p:=(1-\pi_0)\sum_{r=1}^\infty 
\P(\xi^\bullet = r) \P(X_{r+1}^{(r)}\geq1).
\]
The sum equals the probability that a normal black vertex
has outdegree at least 1 and that the last clockwise
increment is at least 1.    
Let $\zeta,\zeta_1,\zeta_2,\dotsc$ be independent
random variables, having the
geometric distribution
$\P(\zeta=k)=p(1-p)^{k-1}$ for $k\geq1$.
Also let
\[
p':=(1-\pi_0)\sum_{r=1}^\infty \P(\xi^\bullet = r)
\P(X_{1}^{(r)}=-1),
\]
and let $\zeta'$ be independent of the $\zeta$:s
with geometric distribution
$\P(\zeta'=k)=p'(1-p')^{k-1}$ for $k\geq1$.
Since $\pi_0<1$ and since we always assume that
$q_i>0$ for some $i\geq2$ we have that 
$p>0$ and $p'>0$.

Finally recall the concept of \emph{stochastic domination}:
a random variable $X$ is stochastically dominated
by a random variable $Y$ if there is a coupling
$\PP$ of $X$ and $Y$ such that $\PP(X\leq Y)=1$.
Let $\xi^\circ_1,\xi^\circ_2$ be independent copies
of $\xi^\circ$, independent also of the $\zeta_i$
and $\zeta'$.
This section is devoted to the following result:
\begin{theorem}\label{deg_thm}
If  $\kappa>0$ the the degree of the root $e_-$ in $M$
is stochastically  dominated by the sum
\begin{equation}\label{degbound_eq}
\zeta'+\sum_{j=1}^{2+\xi^\circ_1+\xi^\circ_2}(1+\zeta_j).
\end{equation}
\end{theorem}
Theorem~\ref{deg_thm} immediately gives that the
degree of the root in $M$ has exponential tails,
and as explained above, Theorem~\ref{thm:recurrence}
therefore follows once we prove Theorem~\ref{deg_thm}.
\begin{proof}
Recall that the root $e_-$ of $M$ is either the same vertex as the
root $\varnothing$ of $\vartheta$ (if $\epsilon=-1$) or it is the
successor of $\varnothing$ (if $\epsilon=+1$).  We will show that in
the first case the degree is bounded by the smaller number
\begin{equation}\label{degbd_1}
\sum_{j=1}^{2+\xi^\circ_1+\xi^\circ_2}(1+\zeta_j),
\end{equation}
and in the second case by~\eqref{degbound_eq}.

We begin with the case $\epsilon=-1$ when $e_-=\varnothing$.
Our argument (and bound on the degree) in this case actually applies 
slightly more generally, to any vertex in $M$ which corresponds to a
vertex `in a fixed position in $\vartheta$'.  We will describe what we
mean by this below.  The argument relies on  a construction of $\vartheta$
which proceeds progressively through the counterclockwise contour 
sequence (see Section~\ref{ss:mobiles}).  We start by
defining a `template' $\vartheta(0)$ of $\vartheta$, by letting
$\vartheta(0)$ denote the mobile obtained by sampling the spine
and all white vertices adjacent to the spine, as well as the
labels of all these white vertices
(subject to $\varnothing$ having label 0, say).  
For convenience we slightly modify  $\vartheta(0)$
by placing a `half-edge' at $\varnothing$,
which  allows us to distinguish between the
`left' and `right' sides of $\varnothing$.
We denote the white contour
sequence of $\vartheta(0)$ by $(c^\circ_i(0))_{i\in\ZZ}$.
This is defined as in Section~\ref{ss:mobiles},
except that (due to the half-edge)
$\varnothing$ is repeated one extra time.  

Here is a summary of the
main idea;  details will follow.  
Let $\eps_1,\eps_2,\dotsc$ be independent Bernoulli
variables taking value 1 with probability $1-\pi_0$, and note
that $\xi^\circ+1$ has the law of the smallest $k$
such that $\eps_k=0$.
The procedure starts 
at some white vertex $w$, and each time
a white vertex, say $v$, is visited the next value $\eps_i$
is examined to determine whether $v$ has `another'
black child.  If so, the number of white children
of this new black vertex is sampled (along with their labels)
and we proceed to the next white vertex in the
counterclockwise contour order.  
This procedure will create the
part of $\vartheta$ which lies after the initial
vertex $w$ in the counterclockwise contour sequence, and
will therefore let us examine the number of white vertices
in $\vartheta$ which have $w$ as their successor.

\begin{figure}[htb]
\begin{center}
\includegraphics{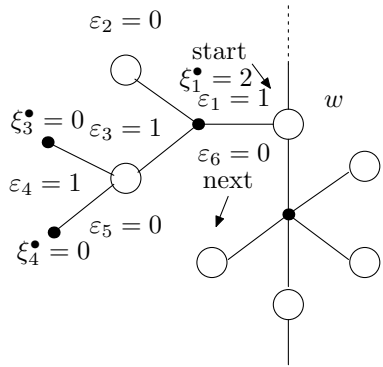}
\caption{The first 6 steps of the construction starting
at a vertex $w$ on the spine.  Labels on white vertices
are not shown.} \label{f:degree}
\end{center}
\end{figure}
Here is a more detailed description,  see also
Fig.~\ref{f:degree} for an illustration.    Let 
$\xi^\bullet_1,\xi^\bullet_2,\dotsc$ be independent,
distributed as $\xi^\bullet$.
We take all the $\zeta'$, $\zeta_i$, $\eps_i$ 
and $\xi^\bullet_i$ to be independent of each other.
We now describe the construction starting 
at a white vertex $w$ of $\vartheta(0)$.  We start our
construction at the \emph{last} visit of
the (clockwise) contour sequence to $w$;  that is, the largest
$i\in\ZZ$ such that $c^\circ_i(0)=w$.  
By shifting the indices, we may (and will) assume
that this smallest index is $i=0$.
We now examine
$\eps_1$.  If $\eps_1=1$  we do the following.
First attach a black vertex to $c^\circ_0(0)$.  
Then  examine the value
of $\xi_1^\bullet$;  if $\xi_1^\bullet=r\geq1$ we attach to this
black vertex $r$ further white vertices.
Finally sample the labels of the new white
vertices by sampling an independent copy of
the discrete bridge $(X_1^{(r)},\dotsc,X_{r+1}^{(r)})$.
We denote the mobile thus obtained by $\vartheta(1)$.
If, on the other hand, $\eps_1=0$ then we just let
$\vartheta(1)=\vartheta(0)$.  We let
$(c^\circ_i(1))_{i\in\ZZ}$ denote the white contour
sequence of $\vartheta(1)$, indexed so that 
$c^\circ_i(1)=c^\circ_i(0)$ for all $i\geq0$. Thus the new white
vertices are placed immediately counterclockwise (in the contour sequence) 
from our starting point $c_0^\circ(0)$.

We now proceed to the next white vertex
$c^\circ_{-1}(1)$ in the counterclockwise contour sequence
of $\vartheta(1)$, and repeat
the same procedure, with $\eps_1$ and $\xi^\bullet_1$
replaced by $\eps_2$ and $\xi^\bullet_2$, respectively.
Note that $c^\circ_{-1}(1)$ is
\begin {itemize}
 \item  the same vertex as $w$ if $\eps_1 = 1$ and $\xi^\bullet_1=0$, or
 \item one of the white vertices just added if $\eps_1 = 1$ 
and $\xi^\bullet_1 >0$, or
 \item  another vertex belonging to the 
template $\vartheta(0)$ if $\eps_1 = 0$.
\end {itemize}
 We thus
obtain a new mobile $\vartheta(2)$ and a new white contour
sequence $(c^\circ_i(2))_{i\in\ZZ}$, which we index so that 
$c^\circ_i(2)=c^\circ_i(1)$ for all $i\geq-1$.  

The procedure
is then carried out inductively.  In this way we obtain 
(in the limit) a mobile whose counterclockwise contour sequence,
started at $w$, agrees in distribution with that
of $\vartheta$.   
We now explain how this construction, started at
the white vertex $w$, allows us to bound the degree
of $w$ in $M$.  By the BDG bijection, the neighbours
of $w$ in $M$ are either successors of $w$, or
have $w$ as a successor.  The successors of
$w$ are easy to count in the procedure above:
each new visit to $w$ corresponds to exactly
one successor.  For a normal white vertex $w$
the number of visits to $w$ has the law 
of $1+\xi^\circ$, while for a special white
vertex the number of visits is the sum of
two independent copies of $1+\xi^\circ$.
(The first corresponding to the visits
to $w$ on the `left side' of the spine, and
the second to the visits on the `right side'.)
Thus, the successors of $w$ account for the
`$1+$' in the summand of~\eqref{degbd_1}.

It remains to count the number of times $w$ 
appears as the successor of other white vertices.  
If $v$ has $w$ as a successor we will call $v$
a \emph{predecessor} of $w$, and we count the predecessors
of $w$ with multiplicity.
By shifting the labels
we may assume that $w$ has label 0, and hence the
predecessors all have label 1.  We  group the 
predecessors $v$ of $w$ by how many times we have
visited  $w$ before we visit $v$.
(This is well-defined as we will never visit
$w$ between visits to $v$.)
The numbers of predecessors in the various groups
are independent, and we claim that the number of
predecessors in each group is stochastically
bounded by $\zeta$.  This will
establish~\eqref{degbd_1}.

In order to establish the claim we define a certain
`stopping event' $A$. 
Suppose at stage $i\geq2$ in the construction
we visit a white vertex $v$ with label
$1$.  Thus $v$ is a potential predecessor of $w$.  
Let $A_i$ be the event that
(i) $\eps_i=1$, (ii) $\xi_i^\bullet=r\geq1$, and (iii)
$X_{r+1}^{(r)}\geq1$.  If this event occurs, then the next vertex
visited in the construction is a recently added vertex with
label 0 or less.  Thus, until we visit $w$ again, any
white vertex we visit which has label 1 \emph{cannot}
have $w$ as successor: there must be a vertex
with label 0 occurring before $v$ in the (clockwise)
contour sequence.  The number of `attempts'
before an event $A_i$ occurs has the distribution of
$\zeta$ (although in general we may get fewer predecessors
in a group since we may return to $w$ before there
is a successful `attempt').  This proves that the
bound~\eqref{degbd_1} applies to any white vertex $w$ in the
`template' $\vartheta(0)$, in particular to $\varnothing$.

Before proceeding to the case $\epsilon=+1$
we note that, although we have assumed that $w$ is a white vertex on or
adjacent to the spine,  it is clear that each time a
white vertex is visited for the first time in the construction above, 
the exact same procedure starts
afresh at that vertex.  Thus we may speak of 
starting the construction at an
\emph{arbitrary} white vertex $w$ of $\vartheta$.
Such a vertex is what was referred to above as a vertex in a `fixed
position in $\vartheta$', and the bound~\eqref{degbd_1}
applies to the degree of any such vertex.  
We will not describe formally what we mean
by `fixed position', but essentially the construction above may be
translated into a deterministic coordinate system in which each white
vertex $w$ has a fixed coordinate.  If $\epsilon=+1$ the root has a
random position in this coordinate system and we require an additional
argument to bound its degree, which we now describe.

We now show that the degree of $e_-$ is bounded 
by~\eqref{degbound_eq} when $\epsilon=+1$, that is when
$e_-$ is the successor of $\varnothing$.
Our description for this case will be slightly less
detailed.
Again we start with the template $\vartheta(0)$.
Denote the half-edge attached to $\varnothing$ by $h$, 
and note that the predecessors of $e_-$ fall into the 
following three categories depending on their
location in $\vartheta$:
\begin{enumerate}
\item\label{rootcase1} those between $h$ and the first occurrence of
$e_-$ in the clockwise contour sequence,
\item\label{rootcase2} those between $h$ and the first occurrence of
a vertex labelled $-1$ in the counterclockwise 
contour sequence, and
\item\label{rootcase3} those which are descendants of
$e_-$ in $\vartheta$.
\end{enumerate}
(The vertex labelled $-1$ in case \eqref{rootcase2} may
be $e_-$ itself if it is a special vertex of $\vartheta$.)
To bound the number of predecessors in category~\eqref{rootcase2}
we may apply a similar scheme as above, constructing the part
of $\vartheta$ counterclockwise from $h$ until the stopping
event $A$ occurs.  That is, each time we encounter a white
vertex $v$ labelled 0 we add this to the list of possible
predecessors of $e_-$, stopping if $v$ has a black child
$u$ of outdegree $r\geq1$ and the first child of
 $u$ in the counterclockwise order has label $-1$ or less.
Thus category~\eqref{rootcase2} has size dominated
by $\zeta$.

To bound category~\eqref{rootcase1} we apply a similar 
scheme, but proceed in the \emph{clockwise} order
starting at $h$.  Each time we visit a white vertex
$v$ it has `another' black child $u$ with probability 
$1-\pi_0$, in which case we sample the outdegree
$r$ of $u$ and then the label increments 
$(X^{(r)}_1,\dotsc,X^{(r)}_{r+1})$ around $u$.  This time we
stop if:  $v$ has label 0, $r\geq1$, and 
$X^{(r)}_1=-1$.  Again, the number of 
white vertices labelled 0 that we encounter before stopping 
is geometrically distributed, 
independent of our bound in case~\eqref{rootcase2},
but this time with parameter $p'$.
Note that we do not stop before encountering $e_-$ in
the clockwise contour sequence, hence the number
in category~\eqref{rootcase1} is dominated by $\zeta'$.

Having thus, in the course of case~\eqref{rootcase1},
located $e_-$, we note that some of the predecessors
in category~\eqref{rootcase3} may have already
been counted in our bounds for 
cases~\eqref{rootcase1} and~\eqref{rootcase2}.
However, we obtain an upper bound on 
category~\eqref{rootcase3} if we assume this not to
be the case.  We may apply a similar argument
as when $e_-=\varnothing$ to deduce that
the number in category~\eqref{rootcase3} is
dominated by $\sum_{i=1}^{\xi_1^\circ+\xi_2^\circ+1}\zeta_i$.
(Here the number of groups to be considered
is dominated by $\xi_1^\circ+\xi_2^\circ+1$ since
one group was already covered by 
cases~\eqref{rootcase1}--\eqref{rootcase2}.)
Finally, taking into account that $e_-$ has
$\xi_1^\circ+\xi_2^\circ+2$ successors,  we arrive
at the bound~\eqref{degbound_eq}.
\end{proof}

\section{The spectral dimension when $\kappa<1$} \label{s:rw}
In this section we focus on the case when $\kappa < 1$, i.e.~when 
$M$ has a face of infinite degree. 
Let $\vartheta = (\tilde{\T},\ell)$ be the  mobile
with distribution $\tilde\mu$, and 
$\rmap=\Phi(\vartheta,\epsilon)$ the corresponding map.
(See Sections~\ref{ss:tildenu}--\ref{ss:labelledmobiles} 
for the definition of $\vartheta$
and Section~\ref{ss:bdg} for $\Phi$.)
By Theorem~\ref{thm:locallimit}, the law of $M$
is $\mu$.   Denote the graph metric
of $\rmap$ by $d$.   The
black vertex of infinite degree in $\vartheta$ will be denoted by $s$
as before.
We will use recent results by Kumagai and 
Misumi~\cite{kumagai:2008} which allow us to calculate the spectral dimension
of $M$ and thereby prove Theorem \ref{th:spec}. Their methods involve
establishing suitable bounds on resistance and volume growth. 
Shortly we will give the necessary definitions
and state the results we need from \cite{kumagai:2008},
but here is a rough outline of the argument.  

Intuitively, our arguments rely on showing that
the resistance and volume growth in $\rmap$ are governed
by $s$, in the sense that if we truncate $\vartheta$
by removing everything except $s$ and its white
nearest neighbours, then the volume and
resistance growth are largely unaffected.  The map
associated via the BDG bijection with the truncated mobile
is the UIPTree (cf.~Theorem~\ref{thm:locallimit}) so the
spectral dimension of $\rmap$ should be that of the UIPTree,
which is 4/3.  For our argument to work, apart from
the main assumption
$\mathbb{E}(\xi) = \kappa<1$ we also need the technical assumption that there exists a $\beta > 5/2$ such that $\mathbb{E}(\xi^\beta) < \infty$. We will assume these properties to hold in the remainder of this section unless otherwise stated.

We begin by discussing some results that allow us to make
precise the idea that the vertex of infinite degree
`dominates' the structure of the mobile $\vartheta$.

\subsection{Decorations}
\label{ss:dec}

Recall that the neighbours of $s$ in
$\vartheta$ are denoted $(s_i)_{i\in \mathbb{Z}}$,
and ordered as in Section~\ref{ss:pt}.
Thus $s_0$ is the parent of $s$, and hence also
the last vertex on the finite spine of $\vartheta$.
For $i\neq0$, the vertex $s_i$ is the root of a subcritical
(modified) Galton--Watson tree
consisting of the descendants of $s_i$;  we call these
trees \emph{decorations} and denote them by $\Dec_i$.
Also $s_0$ may be viewed as the root of a tree,
consisting of all vertices which are not
descendants of $s_0$.  This tree consists of
a spine of geometrically distributed
length with subcritical Galton--Watson trees
attached to it, see Section~\ref{ss:tildenu}.  
We denote this tree by $\Dec_0$
and call it the \emph{bad} decoration
since it is considerably larger than
the other decorations. If $v \in V^\circ(\Dec_i)$ we will 
write $v^\star = s_i$ and $\Dec(v^\star) =\Dec(v) = \Dec(s_i) = \Dec_i$. 

Denote the number of vertices in decoration $i$ by $|\Dec_i| = |V(\Dec_i)|$.
The first lemma explains how the moment condition 
on $\xi$ provides corresponding moment conditions on 
$|\Dec_i|$. A proof is given in the Appendix.
\begin {lemma} \label{l:ximoment}
If $\mathbb{E}(\xi^r) < \infty$ for some $r \geq 1$ 
then $\mathbb{E}(|\Dec_i|^r) < \infty$ for $i\neq 0$ 
and $\mathbb{E}(|\Dec_0|^{r-1})< \infty$.
\end {lemma}

The next lemma is from \cite[Lemma 2.1]{janson:2012sl} and 
relates the maximum displacement of labels in a decoration 
to the number of vertices it contains.
\begin {lemma} \label{l:deltal}
Write
\begin{equation}
\Delta\ell(\Dec_i) = 
\max_{y \in V^\circ(\Dec_i)}|\ell(s_i) - \ell(y)|.
\end{equation}
For any $r>0$ there is a constant $C(r)$ such that
\begin {equation}
\mathbb{E}(\Delta\ell(\Dec_i)^r\mid \Dec_i) 
\leq C(r) |\Dec_i|^{r/2}. 
\end {equation}
\end {lemma}
For easier notation we will sometimes write 
$\Delta\ell_i = \Delta\ell(\Dec_i)$.

\subsection{Resistance and volume}

As mentioned above we will prove Theorem~\ref{th:spec}
by proving bounds on volume and resistance growth
in $M$ and then appealing to the results 
of~\cite{kumagai:2008}.  We first recall the basic 
definitions of electrical networks, 
see e.g.~\cite{lyons-peres} for more details.

Let $G=(V,E)$ be a locally finite graph, with vertex set $V$ and edge
set $E$.  Eventually we will consider $G = M$. 
A \emph{resistance} is a function
$r\colon E\rightarrow[0,\infty]$, and the 
resistance of an edge $e$ will be written as $r_e = r(e)$. The associated
\emph{conductance} function $c\colon E\rightarrow[0,\infty]$
is given by $c_e = c(e)=1/r_e$.  
(The case when some of the $c_e$ are infinite
can be reduced to the case when all the $c_e$ are
finite by identifying adjacent vertices $x, y$ such that
$c_{xy}=\infty$.)
For the conductance $c$ fixed, and
any function $f\colon V\rightarrow\RR$, we define the
\emph{Dirichlet energy}
\begin{equation}
\cE(f):=\sum_{xy\in E} c_{xy}(f(x)-f(y))^2
\end{equation}
(where $\infty\cdot0=0$).
For two disjoint sets $A,B\subseteq V$ we define the
\emph{effective resistance} $\Reff_{(G,c)}(A,B)$
by
\begin{equation}\label{reff_def_eq}
\Reff_{(G,c)}(A,B)^{-1}:=\inf\{\cE(f): f(a)=1\;\forall a\in A,
f(b)=0\;\forall b\in B\}.
\end{equation}
In the case when all $c_e=1$ we also write
$\Reff_G$ for the effective resistance.
We write $\Reff_G(\{x\},\{y\})=\Reff_G(x,y)$ etc.

Let $\vartheta^\star$ be the
truncated mobile obtained from $\vartheta$ by
throwing away all vertices except $s$ and 
$(s_i)_{i\in\mathbb{Z}}$,
keeping the labels of these vertices. 
Let $M^\star = \Phi(\vartheta^\star,\epsilon)$ and
note that $M^\star$ is the UIPTree. The
directed root edge in $M^\star$ obtained by the BDG construction will
be denoted by $(\uiptroot,\uiptroot1)$ and we will use the convention that
$\uiptroot$ is the root of $M^\star$ and 
that $\uiptroot1$ is the leftmost
child of $\uiptroot$. The infinite spine in $M^\star$ will be denoted by $S^\star$ and the vertex in $S^\star$ at a distance $i$ from $\uiptroot$ will be denoted by $S^\star_i$. Denote the graph metric on $M^\star$ by
$d^{\star}$. Note that in the case $\kappa = 0$ we have $M = M^\star$
almost surely.

The results of~\cite{kumagai:2008}
are formulated in terms of an arbitrary metric on the vertex set of the graph under consideration
(not necessarily the graph metric).  Recall that we identify $V(M)$
with $V^\circ(\tilde{\T})$.  We will be using the 
 metric $\dsub$ on  $V^\circ(\tilde{\T})$  defined by
\begin {equation}\label{dsub_def_eq}
\dsub(u,v) = d^\star(u^\star,v^\star) + 
(1-\delta_{u,u^\star}) +
(1-\delta_{v,v^\star}),
\end {equation}
where $\delta_{a,b}$ is the Kronecker delta.  Denote by $M^\#$ the
graph whose vertex set is $V^\circ(\tilde\T)$
and with edges between vertices at $d^\#$-distance one.
It is clear that $M^\#$ is a tree and that it contains
$M^\star$ as a subtree. 
However, $M^\star$ is \emph{not} in general
a subgraph of $M$.
In the following, we will take $\uiptroot$ to be
the reference vertex in $M$, $M^\star$ and $M^\#$ when referring to
graph balls and resistance ($\uiptroot$ corresponds to the vertex 0
defined in \cite[above Eq.~(1.3)]{kumagai:2008}).

For $v \in V^\circ(\tilde\T)$, let $\omega(v)$ be the number of edges adjacent to $v$ in $M$ and extend $\omega$ to a measure on $V^\circ(\tilde\T)$. 
Define
\begin {equation}
 B(R;\dsub) = \{v \in V^\circ(\tilde\T)~:~ \dsub(\uiptroot,v) < R\},
\end {equation}
and write $\omega(R)$ for $\omega(B(R;\dsub))$.

We will be using the following result 
from~\cite{kumagai:2008}.  
For each $\l>1$ define the random set
\begin{equation}
\begin{split}
J(\l)=\{R\in[0,\infty]\colon &
\l^{-1}R^2\leq \omega(R)\leq \l R^2,\\
&\Reff_M(\uiptroot,B(R;\dsub)^c))\geq \l^{-1}R,\\
&\exists y\in B(R;\dsub):\Reff_M(\uiptroot,y)\leq \l \dsub(\uiptroot,y)
\}.
\end{split}
\end{equation}
By~\cite[Theorem~1.5]{kumagai:2008},
if there are $\l_0>1$ and $c,q>0$ such that
\begin{equation}\label{jprob_eq}
\P(J(\l)\ni R)\geq 1-c\l^{-q},\mbox{ for all }R\geq1, \l\geq\l_0,
\end{equation}
then $\specdim(M)=4/3$ almost surely.

Here is an outline of the rest of this section.  
To prove~\eqref{jprob_eq} we will
treat each of the four inequalities defining $J(\l)$
separately, in a sequence of lemmas.  Bounds on the
volume $\omega(R)$ will be treated in Section~\ref{ss:vol}
(Lemmas~\ref{lem:uvol} and~\ref{lem:lvol}).
In Section~\ref{ss:ures} we establish an upper bound
on the probability that 
$\exists y \in B(R;\dsub): \Reff_M(\uiptroot,y)> \l \dsub(\uiptroot,y)$
(Lemma~\ref{lem:ures}).  Finally, in Section~\ref{ss:lres}
we deal with the hardest part of the argument, which is 
to establish an upper bound on the probability that
$\Reff_M(\uiptroot,B(R;\dsub)^c))< \l^{-1}R$
(Lemma~\ref{lem:lres}).  For this result we will use a
technique of `projecting long bonds',
inspired by methods previously applied to
 long-range percolation models in~\cite[Proposition~2.1]{kumagai:2008}
and~\cite[Lemma~3.8]{berger:2002}.

\subsection{Bounds on the volume}
\label{ss:vol}

We begin with the upper bound on the volume. 
We will need two preliminary lemmas. First we state the 
following result on the
labels in $\vartheta^\star$ which will also be used
in the bounds on resistance growth. A proof is given in the Appendix.

\begin{lemma}\label{kmt_lem}
Assume $\ell(s_0) = 0$ and define
\begin {eqnarray}
 i^+(R) &=& \inf\{i \geq 0~:~\ell(s_i) = -R\} \quad \text{and}\\
i^-(R) &=& \inf\{i \geq 0~:~\ell(s_{-i}) \leq -R\}. 
\end {eqnarray}
There is a constant $C>0$ such that for all
$R,\l\geq1$ we have
\[
P(i^\pm(R)>\l R^2)\leq C\log (\l)\l^{-1/2}.
\]
\end{lemma}
The second lemma bounds the minimum label in $B(R;\dsub)$.
\begin {lemma} \label{l0}
Define
\begin {equation}
 m(R) = \min\Bigg\{\ell(v) ~:~ v \in \bigcup_{-i^-(R-1) < i \leq i^+(R)} V^\circ(\Dec_i) \Bigg\}.
\end {equation}
Let $r>1$ and define $\alpha = \min\{2r-2,2 r/3\}$.  If $\mathbb{E}(\xi^r) < \infty$ then there is a constant $c>0$ and a $\lambda_0 > 1$ such that
\begin {equation}
 \P(|m(R)| > \lambda R) < c \log(\lambda)\lambda^{-\alpha}
\end {equation} 
for every $R\geq 1$ and $\lambda > \lambda_0$.

\end {lemma}
\begin {proof}
Define $N(x) = i^+(x) + i^-(x-1)-1$. We have for any $i$ satisfying $-i^-(R-1) < i \leq i^+(R)$ that
\begin {equation}
 \min\{\ell(v)~:~v\in V^\circ(\Dec_i)\} \geq \ell(s_i) - \Delta\ell_i \geq -R - \Delta\ell_i
\end {equation}
 and therefore
\begin {equation}
 |m(R)| \leq R + \max\{\Delta\ell_i~:~ 
-i^-(R-1) < i \leq i^+(R)\}. 
\end {equation}
Thus, for any $\gamma > 0$
\begin {align} \nonumber 
  \P(|m(R)| > \lambda R) &\leq \P\left(\max\{\Delta\ell_i^{2r}~:~ 0 \leq i \leq N(R)\} > ((\lambda -1)R)^{2r}\right) \nonumber \\
&\leq \P\Bigg(\sum_{i=0}^{\gamma R^2}\Delta\ell_i^{2r}> ((\lambda -1)R)^{2r}\Bigg)+\P(N(R)>\gamma R^2) . \nonumber \\ \label {mr}
\end {align}
In the first term in the last expression, we treat separately the contribution of the bad decoration and the sum of the others. For any $0 < \delta < 1$ the first term may be bounded by
\begin {align} \nonumber
 \P\Bigg(\sum_{i=1}^{\gamma R^2}\Delta\ell_i^{2r}> (1-\delta)((\lambda -1)R)^{2r}\Bigg) &+ \P(\Delta\ell_0^{2r-2} > \delta^{\frac{r-1}{r}} ((\lambda -1)R)^{2r-2}) \\
 &\leq \frac{\gamma \mathbb{E}(\Delta\ell_1^{2r})}{(1-\delta)(\lambda-1)^{2r}}+ \frac{\mathbb{E}(\Delta\ell_0^{2r-2})}{\delta^{\frac{r-1}{r}}(\lambda-1)^{2r-2} } 
\end {align}
 where we used Markov's inequality on both terms along with $R^{2-2r} \leq 1$. By Lemmas \ref{l:ximoment} and \ref{l:deltal}, both expected values in the last expression are finite since $\mathbb{E}(\xi^r) < \infty$. Applying Lemma \ref{kmt_lem} to the second term on the right hand side of \eqref{mr} we finally obtain
 \begin {equation} \label{mrfinal}
  \P(|m(R)| > \lambda R) \leq c_1 \gamma \lambda^{-2r}+c_2 \lambda^{-2r+2}+c_3 \log(\gamma)\gamma^{-1/2},
 \end {equation}
 where $c_1,c_2,c_3 >0$ are constants. The choice $\gamma = \lambda^{4r/3}$ optimizes the inequality.
\end {proof}
Finally, we arrive at the upper bound on the volume.
\begin{lemma}\label{lem:uvol}
Let $r>1$ and define $\alpha' = \min\big\{\frac{r(r-1)}{2r^2-1},\frac{r}{2(r+2)}\big\}$. If $\mathbb{E}(\xi^r) < \infty$ then there is a constant $c'>0$ and a $\lambda_0 > 1$ such that  
\begin {equation}
 \P(\omega(R) > \lambda R^2) \leq c' \log(\l)\lambda^{-\alpha'}
\end {equation}
for every $R\geq 1$ and $\lambda \geq \lambda_0$.
\end {lemma}

\begin {proof}
We argue on the event when $\epsilon = -1$ in which case $\ell(\uiptroot) = 0$. The other case is treated in a similar way, the difference being that $\ell(\uiptroot) = -1$ and therefore one has to shift labels accordingly in the arguments.
 Define
\begin {equation}
 H(R) = \bigcup_{-i^-(|m(R-1)|) < j \leq i^+(|m(R-1)|+1)} V^{\circ}(\Dec_j).
\end {equation}
We claim that a vertex in the set $B(R ; d^\#)$ is not connected by an edge in $M$ to any vertex outside the set $H(R)$. To see this, first observe that
\begin {equation}
\min\{\ell(v)~:~v\in B(R;\dsub)\} \geq m(R-1).
\end {equation} 
The successor of a vertex
in $B(R;\dsub)$ as defined in (\ref{eq:succ}) can therefore not be
in the part of the white contour sequence strictly beyond the vertex
$s_{i^+(|m(R-1)|+1)}$ (the connection has to `go through' this
vertex). Similarly, no vertex before $s_{-i^-(|m(R-1)|)}$ in the
white contour sequence can have a successor in $B(R;\dsub)$ which
completes the proof of the claim. 

This implies that $|H(R)| \geq \omega(R)$ and thus
\begin {equation}
 \P(\omega(R) > \lambda R^2) \leq \P(|H(R)| > \lambda R^2).
\end {equation} As before, let $N(x) = i^+(x) + i^-(x-1)-1$. Then for any $\gamma > 0$
\begin {equation}\label{hestimate}
 \P(|H(R)| > \lambda R^2) \leq \P(N(|m(R-1)|) > \gamma R^2) + 
\P\Big(\sum_{i=0}^{\gamma R^2} |\Dec_i| > \lambda R^2\Big).
\end {equation}
Since $i^+$ and $i^-$ are increasing functions, then for any $\eta >0$ the first term may be estimated from the above by
\begin {eqnarray} \nonumber
 \P(N(|m(R-1)|) > \gamma R^2) &\leq& \P(N(\eta R) > \gamma R^2) +  \P(|m(R-1)| > \eta R) \\ &\leq& c_1 \log(\gamma/\eta^2)\eta \gamma^{-1/2}+c_2 \log(\eta)\eta^{-\alpha}
\end {eqnarray}
where $c_1,c_2>0$ are constants and $\alpha = \min\{2r-2,2r/3\}$. In
the last step, the estimate of the first term was obtained using Lemma
\ref{kmt_lem} and the second estimate came from using Lemma
\ref{l0}. The second term on the right hand side of (\ref{hestimate})
is estimated by separating the bad decoration from the rest as in the
proof of Lemma \ref{l0} and we obtain for any $0 < \delta < 1$
\begin {align*}
 \P\Big(\sum_{i=0}^{\gamma R^2} |\Dec_i| > \lambda R^2\Big) 
&\leq \P\Big(\sum_{i=1}^{\gamma R^2} |\Dec_i| > (1-\delta)\lambda R^2\Big) 
+ \P(|\Dec_0| > \delta \lambda R^2) \\
 &\leq \frac{\mathbb{E}\big(\sum_{i=1}^{\gamma R^2} |\Dec_i|\big)^r}
{(1-\delta)^r\lambda^r R^{2r}} + 
\frac{\mathbb{E}(|\Dec_0|^{r-1})}{\delta^{r-1}\lambda^{r-1}R^{2(r-1)}} \\
 &\leq \frac{\gamma^r \mathbb{E}(|\Dec_1|^r)}{(1-\delta)^r\lambda^r} + 
\frac{\mathbb{E}(|\Dec_0|^{r-1})}{\delta^{r-1}\lambda^{r-1}}
\end {align*}
where we used Markov's inequality and then Minkowski's inequality. By
Lemma \ref{l:ximoment} both expected values in the last expression are
finite since $\mathbb{E}(\xi^r) < \infty$. Thus, we finally have
\begin {align}
  \P(|H(R)| > \lambda R^2) \leq c_1 \log(\gamma/\eta^2)\eta
  \gamma^{-1/2}+c_2 \log(\eta)\eta^{-\alpha} + c_3 \gamma^r
  \lambda^{-r} + c_4 \lambda^{-r+1}
\end {align}
where $c_3,c_4 > 0$ are constants.  Choosing $\eta$ as a positive
power of $\gamma$ and $\gamma$ as a positive power of $\lambda$ allows
one to deduce that the exponent $\alpha'$ gives the optimal bound.
\end {proof}

The following result gives the lower bound
on the volume.

\begin{lemma}\label{lem:lvol}
There exists a $\lambda_0>1$  and a constant $c > 0$ such that
\begin {equation}
 \P(\omega(R) < \lambda^{-1} R^2) \leq e^{-c \lambda^{1/2}}
\end {equation}
for every $R\geq 1$ and $\lambda \geq \lambda_0$.
\end{lemma}
\begin {proof}
 The result follows from comparing $B(R;\dsub)$ to the graph ball in $M^\star$
\begin {equation}
B(R;d^\star) :=  \{v \in V(M^\star)~:~ d^\star(\uiptroot,v) < R\}.
\end {equation}
From the definition of $\omega$ and $\dsub$ 
it follows that 
$\omega(R) \geq |B(R;\dsub)| \geq |B(R;d^\star)|$ and one has the bound
\begin {equation}
 \P(|B(R;d^\star)| < \lambda^{-1} R^2) \leq e^{-c \lambda^{1/2}}
\end {equation}
for the UIPTree $M^\star$, see e.g.~\cite{durhuus:2007,fuji:2008}.
\end{proof}

\subsection{Upper bound on the resistance}
\label{ss:ures}

For any vertex $v$ in $M$ define the successor geodesic from $v$ to infinity as 
$\succg(v) =(v,\sigma(v),\sigma(\sigma(v)),\ldots)$ 
where $\sigma$ is defined in (\ref{succv}). 
For vertices $u,v$ in $M^\star$ denote the unique geodesic (i.e.~shortest path in $M^\star$) between $u$
and $v$ by $\geo^\star(u,v)$ and denote the successor geodesic of $v$
in $M^\star$ by $\geo^\star(v)$. Let $\langle v, u\rangle$ be the
vertex in $\gamma^\star(v)\cap \gamma^\star(u)$ closest to $v$ (and
$u$) in $M^\star$. If $u = s_i$ and $v = s_j$, $i\leq j$, define the closed
interval $[u,v] = \{s_i,s_{i+1},\ldots,s_j\}$. 
Furthermore, define $s_i\wedge  s_j = s_{i\wedge j}$. 
The next lemma demonstrates how close the metric $d$ is to
 the metric $d^\star$ measured in terms of the size of the
 maximum displacement of labels in the decorations. 
It resembles Lemma 5.2 in \cite{janson:2012sl} with a very similar 
proof which we give in the Appendix.

\begin {lemma}  \label{l:geosup}
For all $v,w \in V(M)$
\begin {equation}
 d(v,w) \leq d^\star(v^\star,w^\star)+ 
20 \max \{\Delta\ell(\Dec(u)):
u\in [v^\star\wedge w^\star,\langle v^\star,w^\star \rangle]\} 
+ 8.
\end {equation}
\end {lemma}
We are now equipped to establish an upper bound on the resistance provided in the following lemma.
\begin {lemma}\label{lem:ures}
Let $r>1$ and define $\alpha = \min\{2r-2,2 r/3\}$.  If $\mathbb{E}(\xi^r) < \infty$ then there is a constant $c(r)>0$ and a $\lambda_0>1$  such that
\begin {equation} \label{upres}
 \P\big(\exists v \in B(R;\dsub): \Reff(\uiptroot,v) > \lambda d^\#(\uiptroot,v)\big) < c(r) \log(\lambda)\lambda^{-\alpha'}
\end {equation}
for every $R\geq 1$ and $\lambda \geq \lambda_0$.
\end {lemma}
\begin {proof}

First observe that $d(\uiptroot,v) \geq \Reff(\uiptroot,v)$ for all $v \in
V(M)$. We may rule out the case that the vertex $v$ in \eqref{upres} is equal to $\uiptroot$ and we note that  $d^\#(\uiptroot,v) \geq \max\{d^\star(\uiptroot,v^\star),1\}$ when  $v\neq \uiptroot$. Using these facts together with Lemma
\ref{l:geosup} one can thus estimate the probability in (\ref{upres})
from the above by
\begin {equation} \label{up1}
 \P\left(\max_{v \in B(R;\dsub)} \frac{20\max \{\Delta\ell(\Dec(u)) :
   u\in [\uiptroot\wedge v^\star,\langle \uiptroot, v^\star\rangle]
   \}}{\max\{d^\star(\uiptroot,v^\star),1\}} > \lambda - 9\right).
\end {equation}
Since $B(R;\dsub)$ is finite, the outermost maximum is attained at some
vertex, say $\bar{v}(R)$ (which may clearly be chosen to be in
$M^\star$).  Then, for any $\gamma >0$ we may estimate (\ref{up1}) from the above by
\begin {align} \nonumber 
&\P\left(\frac{20^{2r} \max\{\Delta\ell_i^{2r} : 0\leq i \leq \gamma \max\{d^\star(\uiptroot,\bar{v}(R)),1\}^2\}}{\max\{d^\star(\uiptroot,\bar{v}(R)),1\}^{2r}}>(\lambda-9)^{2r}\right) \\
&\qquad+\P\left(|[\uiptroot\wedge \bar{v}(R),\langle \uiptroot, \bar{v}(R)\rangle]|-1 > \gamma \max\{d^\star(\uiptroot,\bar{v}(R)),1\}^2\right). \label{up2}
\end {align}
The expression in the first line of (\ref{up2}) may be estimated from the above by first estimating the maximum in the numerator by the sum of all the terms and then separating the bad decoration (first term in the sum) from the rest exactly as was done in the estimate of \eqref{mr} in the proof of Lemma \ref{l0}.  Conditioning on $d^\star(\uiptroot,\bar{v}(R))$ and using Lemma \ref{l:deltal} along with similar arguments as in Lemma \ref{l0} yields the following upper bound on the first term
\begin {eqnarray} \label{urest1}
c_1 \gamma \lambda^{-2r}+c_2 \lambda^{-2r+2}.
\end {eqnarray}
where $c_1,c_2>0$ are constants.

We may use Lemma \ref{kmt_lem} to estimate the expression in the second line of (\ref{up2}). In the case $\epsilon = -1$, writing $d^\star(\uiptroot,\langle \uiptroot, \bar{v}(R)\rangle) = D$, one has
\begin {equation}
 |[\uiptroot\wedge \bar{v}(R),\langle \uiptroot, \bar{v}(R)\rangle]| \leq i^+(D) + i^{-}(D)
\end {equation}
and $\max\{d^\star(\uiptroot,\bar{v}(R)),1\} \geq D$. This yields the upper bound,
\begin {equation} \label{urest2}
 \P\left(i^+(D) + i^{-}(D) -1> \gamma D^2\right) \leq c_3 \log(\gamma) \gamma^{-1/2}. 
\end {equation}
where $c_3>0$ is a constant. The case $\epsilon = 1$ is treated in the same way taking into account that labels are shifted. Combining the estimates \eqref{urest1} and \eqref{urest2} and choosing $\gamma = \lambda^{4r/3}$ gives an optimal upper bound on \eqref{up2} and yields the exponent $\alpha$.
\end {proof}

\subsection{Lower bound on the resistance}
\label{ss:lres}

The only remaining bound needed for~\eqref{jprob_eq},
and hence for proving Theorem~\ref{th:spec}, is the lower
bound on the resistance.  This is given by 
the following lemma.
\begin {lemma}\label{lem:lres}
Let $\gb>5/2$ and 
assume $\E(\xi^\gb)<\infty$.
For $0<q<\min(1,2\gb-5)$  there exists a 
constant $c(q)>0$ and a $\lambda_0>1$  such that
\begin {equation}
 \P(\Reff_M(\uiptroot,B(R;\dsub)^c) < \lambda^{-1} R) 
\leq c(q) \lambda^{-q}
\end {equation}
for every $R\geq 1$ and $\lambda \geq \lambda_0$.
\end {lemma}

To prove this lemma
we will compare resistances in
$M$ with resistances in the tree $M^\#$ with certain
non-constant conductances.  In this section we write 
$e=uv\in M$ if $e$ is an edge of $M$ with endpoints $u$ and $v$.  
Define $|e|_\star=d^\star(u^\star,v^\star)$, and for $xy$
an edge of $M^\star$ let $A_{xy}$
be the set of edges $e=uv$ of $M$ such that 
$xy$ lies on the (unique) path from $u^\star$ to 
$v^\star$ in $M^\star$.  For an edge $xy$
of $M^\#$, define
\begin{equation}\label{c_def_eq}
c_{xy}=\left\{
\begin{array}{ll}
\sum_{e\in A_{xy}} |e|_\star, & \mbox{if } x, y\in M^\star,\\
\oo, & \mbox{otherwise}.
\end{array}\right.
\end{equation}
We claim that for all $A\se V(M)=V(M^\#)$,
\begin{equation}\label{Reff_comp_eq}
\Reff_M(\uiptroot,A)\geq \Reff_{(M^\#,c)}(\uiptroot,A).
\end{equation}
To see this, we note that the network $(M^\#,c)$
can be obtained by modifying $M$ according
to the following procedure.  Let $e=uv$ be an arbitrary
edge of $M$ and for convenience assume that $e$ is directed from $u$
to $v$. If $u^\star = v^\star$ we `short' $e$ by identifying $u$ and
$v$. 
Otherwise, subdivide $e$ into $|e|_\star$ series resistors each with
resistance $1/|e|_\star$ so that the total resistance is still 1.
Then `short' this network by identifying the origin of the first
resistor with $u^\star$ (if $u\neq u^\star$) and by identifying the
endpoint of the $j$th resistor with the endpoint of the $j$th step on
the geodesic from $u^\star$ to $v^\star$ in $M^\star$ (the last
identification is only necessary if $v \neq v^\star$), see
Fig.~\ref{f:short}.
\begin{figure} [t]
\centerline{\scalebox{0.35}{\includegraphics{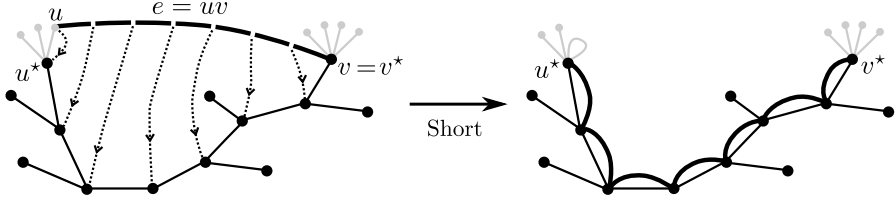}}}
\caption{A part of $M^\#$ is shown on the left where the black solid part is $M^\star$ and the decorations from vertices $u$ and $v$ are shown in gray. The thick dashed edge $e=uv$ is an edge in $M$. The dotted lines represent how the edge $e$ is projected down on the geodesic from $u^\star$ to $v^\star$ by shorting.} \label{f:short}
\end{figure}
By the series- and parallel laws, this gives the
conductance~\eqref{c_def_eq};
by the `shorting law' we obtain~\eqref{Reff_comp_eq}.

By~\eqref{Reff_comp_eq} we have
\begin{equation}
\Reff_M(\uiptroot,B(R;d^\#)^c)\geq 
\Reff_{(M^\#,c)}(\uiptroot,B(R;d^\#)^c)
\geq\Reff_{(M^\star,c)}(\uiptroot,B(R-1;d^\star)^c).
\end{equation}
The last inequality holds because the boundary
of $B(R-1;d^\star)$ separates $\uiptroot$
from the boundary of $B(R;\dsub)$.
This estimate is not useful when $R=1$ but one may easily treat that 
case separately by noting that $B(1;d^\#)$ consists of $S^\star_1$ and 
the vertices in the decoration $\Dec(\uiptroot)$.
In what follows we will be considering
resistances of the form
$\Reff_{(M^\star,c)}(\uiptroot,B(R,d^\star)^c)$.
Our objective will be to show that for $0<q<\min(1,2\gb-5)$ 
there is a constant $c>0$ such that
\begin{equation}\label{eq:R_obj}
\EE[\Reff_{(M^\star,c)}(\uiptroot,B(R;d^\star)^c)^{-q}]
\leq c R^{-q}.
\end{equation}
This will prove Lemma~\ref{lem:lres} since by
Markov's inequality
\begin{equation}
\begin{split}
\P(\Reff_M(\uiptroot,B(R;\dsub)^c) < \lambda^{-1} R) &\leq
\P(\Reff_{(M^\star,c)}(\uiptroot,B(R-1;d^\star)^c)^{-q}
> \lambda^{q} R^{-q}) \\
& \leq
\frac{\EE[\Reff_{(M^\star,c)}(\uiptroot,B(R-1;d^\star)^c)^{-q}]}
{\l^q R^{-q}}.
\end{split}
\end{equation}

We start by finding a vertex $z_R$ in $M^\star$ which separates
$\uiptroot$ from the boundary of $B(R;d^\star)$.
Recall that $S^\star$ is the spine in $M^\star$.
Denote the subgraph of $M^\star$ consisting of $S^\star_i$ 
and the collection of finite outgrowths from the 
normal children of $S^\star_i$ by $W^\star_i$.
Thus $W^\star_i$ is a tree, and we denote the collection
of vertices in generation $n$ of $W^\star_i$
by $W^\star_i(n)$ (where $W^\star_i(0)=\{S_i^\star\}$).
Let $z_R=S^\star_\ell$ for $0\leq \ell\leq R/2$ chosen maximal
such that for all $j<\ell$, we have
$W^\star_j(R-j)=\varnothing$  (ie,
the tree
$W^\star_j$ does not reach level $R$ in $M^\star$).
Note that 
every path from $\uiptroot$ to 
level $R$ in $M^\star$ goes through $z_R$. 

Define $L_R = d^\star(\uiptroot,z_R)$
(thus $L_R=\ell$ in the above). 
The distribution of $L_R$ is easy to compute. Recall the construction of $\sgt$ in Section \ref{s:sgt} and that $\sgt \sim M^\star$ when $w_i = 1$ for all $i$.
Let $(Y_n)_{n\geq 0}$ denote a random sequence with the same
distribution as $(|W^\star_i(n)|)_{n\geq 0}$.  Thus $Y_n$ is
the size of the $n$:th generation
in the modified Galton--Watson process defined as follows.
Firstly, $Y_0=1$ and the zeroth generation offspring
distribution is given by $\P(Y_1=i)=(i+1)\pi_{i+1}$
where $\pi_i=2^{-i-1}$.  For all later generations
the offspring distribution is $(\pi_i)_{i\geq0}$. We collect a few facts about $(Y_n)_{n\geq}$ in the Appendix.
Using \eqref{yn0}, we have for $k < \lfloor R/2 \rfloor$ that
\begin {equation} \label{lrdist1}
 \P(L_R = k) = \P(Y_{R-k} >0)\prod_{j=0}^{k-1} \P(Y_{R-j}=0) = 
{\frac {2\,R-2\,k+1}{ \left( R+1 \right) ^{2}}}
\end {equation}
and 
\begin {equation} \label{lrdist2}
 \P(L_R = \lfloor R/2 \rfloor) = \prod_{j=0}^{\lfloor R/2 \rfloor-1} \P(Y_{R-j}=0) = 
\left(\frac{R-\lfloor R/2 \rfloor + 1}{R+1}\right)^2.
\end {equation}

Turning now to~\eqref{eq:R_obj}, we
begin by considering the case $L_R>0$.  Since
$z_R$ separates $\uiptroot$ from $B(R;d^\star)^c$ we have
\begin{equation}
\begin{split}
\Reff_{(M^\star,c)}(\uiptroot,B(R,d^\star)^c)\indic{L_R>0}
&\geq \Reff_{(M^\star,c)}(\uiptroot,z_R)\indic{L_R>0}\\
&=\sum_{k=1}^{L_R} c(S^\star_{k-1},S^\star_k)^{-1}\indic{L_R>0},
\end{split}
\end{equation}
where in the second step we used the series law.
Let $q\in(0,1)$ and write $A_k=A_{S^\star_{k-1},S^\star_k}$
for the set of edges $uv$ of $M$ such that
the edge $(S^\star_{k-1},S^\star_k)$ of $M^\star$
lies on the path from $u^\star$ to $v^\star$ in $M^\star$. 
It follows that
\begin{equation}\label{eq:LRpos}
\begin{split}
\EE[\Reff_{(M^\star,c)}&(\uiptroot,B(R,d^\star)^c)^{-q}\indic{L_R>0}]
\leq \EE\Big[\Big(\sum_{k=1}^{L_R}c(S^\star_{k-1},S^\star_k)^{-1}\Big)^{-q}
\indic{L_R>0}\Big]\\
&\leq \EE\Big[L_R^{-q-1}\sum_{k=1}^{L_R}
\Big(\sum_{e\in A_k} |e|_\star\Big)^q\indic{L_R>0}\Big]\\
&\leq \EE\Big[L_R^{-q-1}\sum_{k=1}^{L_R}
\EE\Big[\sum_{e\in A_k} |e|_\star^q\,\big|\,L_R\Big]\indic{L_R>0}\Big].
\end{split}
\end{equation}
Here we used H\"older's inequality in the second
step and subadditivity in the third step, before conditioning
on $L_R$.  On the event $L_R>0$ we have that 
\begin{equation}\label{eq:A_sum1}
\EE\Big[\sum_{e\in A_k} |e|_\star^q\,\big|\,L_R\Big]
=\sum_{n\geq1}n^q \EE\big[A_k^{(n)}\mid L_R\big]
\end{equation}
where $A_k^{(n)}$ denotes the number
of edges $e=uv\in A_k$ such that 
$|e|_\star=d^\star(u^\star,v^\star)=n$.
We aim to show that (still on the event $L_R>0$)
the sum in~\eqref{eq:A_sum1} is bounded by a constant.
It is easy to check, using \eqref{lrdist1} and \eqref{lrdist2}, that $\EE[L_R^{-q}\indic{L_R>0}]$
is of the order $R^{-q}$ for $q\in(0,1)$.
Thus
finiteness of~\eqref{eq:A_sum1} will give~\eqref{eq:R_obj}
in the case $L_R>0$.

Now we turn to the case $L_R=0$.
Writing $x\sim y$ if $x,y$ are adjacent 
vertices in $M^\star$, we can use the simple
fact that
\[
\Reff_{(M^\star,c)}(\uiptroot,B(R,d^\star)^c)\geq
\Reff_{(M^\star,c)}(\uiptroot,B(1,d^\star)^c)
=\Big(\sum_{x\sim\uiptroot}c_{\uiptroot,x}\Big)^{-1}.
\]
We have that
\begin{equation}
\begin{split}
\EE[\Reff_{(M^\star,c)}&(\uiptroot,B(R,d^\star)^c)^{-q}\indic{L_R=0}]
\leq \EE[\Reff_{(M^\star,c)}(\uiptroot,B(1,d^\star)^c)^{-q}\indic{L_R=0}]\\
&\leq \P(L_R=0)\EE\Big[\sum_{x\sim\uiptroot}
\sum_{e\in A_{\uiptroot x}}|e|_\star^{q}
\Big| L_R=0\Big]\\
&=\P(L_R=0)\sum_{n\geq1}n^{q}
\EE\Big[\sum_{x\sim\uiptroot}A_{\uiptroot x}^{(n)}\Big| L_R=0\Big].
\end{split}
\end{equation}
Here $A_{\uiptroot x}^{(n)}$ denotes the number
of edges $e=uv\in A_{\uiptroot x}$ such that 
$|e|_\star=d^\star(u^\star,v^\star)=n$.
Since $\P(L_R=0)$ is of order $R^{-1}$ this will
establish~\eqref{eq:R_obj} in the case $L_R=0$ provided
we show that the sum
\begin{equation}\label{eq:A_sum2}
\sum_{n\geq1}n^{q}
\EE\Big[\sum_{x\sim\uiptroot}A_{\uiptroot x}^{(n)}\Big| L_R=0\Big]
\end{equation}
is bounded by a constant.

We thus need bounds on $A_k^{(n)}$ and
$A_{\uiptroot x}^{(n)}$.  
The strategy will be to bound  firstly
the number of edges in $M$
from the decoration $\Dec_i$ to the decoration $\Dec_j$, 
and secondly the number 
of pairs $i$ and $j$ such that a given edge 
lies on the path from $s_i$ to
$s_j$ in $M^\star$.  This is the contents of the
following two lemmas.

For the first lemma we recall that, for $x,y\in M^\star$,
$\langle x,y\rangle$ denotes the first vertex
of $M^\star$ on the successor geodesics of 
both $x$ and $y$.
Recall the definition of a corner, just 
below~\eqref{contourseq}. 
Denote the set of corners around white vertices in
a decoration $\Dec_i$ by $\mathcal{C}_i$ and note that
$|\mathcal{C}_i| = |\Dec_i|$.
 For two corners $u$ and $v$ we write $uv\in M$
if there is an edge in $M$ 
between the vertices corresponding to $u$ and $v$. 
\begin{lemma}\label{geod_prob_est_lem}
Let $s_i,s_j\in M$ be such that $i<j$,
$n:=d^\star(s_i,s_j)\geq 2$, and
write $m:=d^\star(s_i,\langle s_i,s_j\rangle)$.
For any $r,r'>0$  there
is a constant $C>0$ such that
\[
\E[\#\{u\in \mathcal{C}_i, v\in \mathcal{C}_j: uv\in M\}\mid M^\star]
\leq C 
\frac{\EE(|\Dec_i|^{r/2+1})}{(m-1+\gd_{m,1})^{r}}
\frac{\EE(|\Dec_j|^{r'/2})}{(n-m+\gd_{m,n})^{r'}}.
\]
\end{lemma}

To simplify the counting of geodesics in $M^\star$,
we use the convention that the geodesic between
$s_i$ and $s_j$ is directed from $i$ to $j$ if
and only if $i<j$.
Writing $n=d^\star(s_i,s_j)$ and 
$m=d^\star(s_i,\langle s_i,s_j\rangle)$,
note that $0 < m\leq n$ and that the 
labels will first decrease for the first $m$ steps from $s_i$ to $s_j$ 
and increase for the  remaining $n-m$ steps.
Let $\Gamma_{n,m}^{(k)}$ be the number of  (directed)
geodesics which (i) contain the edge $(S^\star_{k-1},S^\star_k)$, 
(ii)  are of length $n$, and (iii) have decreasing labels 
on exactly the first $m$ steps. 

Note that in the definition of $\Gamma_{n,m}^{(k)}$,
the endpoints $s_i$ and $s_j$ are not
fixed. We will consider separately the case when one of the
endpoints is fixed to be $s_0$,
and thus is the root of the bad decoration. The reason 
is that the optimal exponent $r$ or $r'$ in 
Lemma~\ref{geod_prob_est_lem} is worse when the 
corresponding decoration is
the bad decoration $\Dec_0$, but this will be countered by the fact that the
number of geodesics is smaller when one of the endpoints 
is fixed. We define $\hat{\Gamma}_{n,m}^{(k,\alpha)}$, $\alpha
\in\{1,2\}$, to be the number of directed geodesics satisfying
(i)--(iii) and in addition that $s_i = s_0$ (resp.~$s_j = s_0$) if
$\alpha = 1$ (resp.~$\alpha = 2$).

\begin {lemma}\label{num_geod_est_lem}
There is a  constant $c>0$  such that
for any $\ell\geq 1$ and $0\leq k\leq \ell$ 
we have
\begin {equation}\label{eq:geo_free}
\E(\Gamma_{n,m}^{(k)}~|~L_R=\ell) \leq cn^2,
\end {equation}
\begin {equation} \label{eq:geo_i0}
\E(\hat{\Gamma}_{n,m}^{(k,1)}~|~L_R=\ell) 
\leq c(1 + (n-m) \indic{m\in \{\ell,\ell+1\}} ),
\end {equation}
\begin {equation} \label{eq:geo_j0}
\E(\hat{\Gamma}_{n,m}^{(k,2)}~|~L_R=\ell) 
\leq c(1+  m \indic{n-m\in\{\ell,\ell+1\}}).
\end {equation}
\end {lemma}

The same bounds apply to 
$\E(\sum_{x\sim\uiptroot}\Gamma_{n,m}^{(\uiptroot x)}\mid L_R=0)$
but we leave the details to the reader.
One small difference in the proof
is that one must use the
fact that the number of neighbours of $\uiptroot$
in $M^\star$ has finite second moment, but
apart from this it is very similar to the
proof of Lemma~\ref{num_geod_est_lem}.

Before proving Lemmas~\ref{geod_prob_est_lem} 
and~\ref{num_geod_est_lem} we show how they 
give~\eqref{eq:R_obj} and hence Lemma~\ref{lem:lres}.

\begin{proof}[Proof of~\eqref{eq:R_obj}]
We need to show that the two sums~\eqref{eq:A_sum1}
and~\eqref{eq:A_sum2} are bounded by constants.
 We start with~\eqref{eq:A_sum1}, assuming until
further notice that $L_R>0$.  Clearly 
$\E(A^{(n)}_k\mid L_R)=\E\big(\E(A^{(n)}_k \mid M^\star)\mid L_R\big)$. 
We can write
\begin{equation}\label{pij_eq}
\E( A^{(n)}_k \mid M^\star)=\sum_{m=1}^n
\sideset{}{^{'}}\sum_{i<j}
\E[\#\{u\in \mathcal{C}_i, v\in \mathcal{C}_j: uv\in M\}\mid M^\star],
\end{equation}
where the primed sum is over all integers $i<j$ such that
\begin{itemize}
\item $d^\star(s_i,s_j)=n$,
\item the edge $(S^\star_{k-1},S^\star_k)$ lies
on the geodesic from $s_i$ to $s_j$ in $M^\star$, and
\item the labels on the first $m$ steps on this geodesic
are decreasing.
\end{itemize}
Note that  
Lemma~\ref{num_geod_est_lem} bounds the number of terms in this sum,
whereas Lemma~\ref{geod_prob_est_lem} bounds
the summand.

First note that we may assume that $n\geq2$:
for $n=1$ the sum~\eqref{pij_eq} consists of a 
single term which we may bound by a constant.
For $n\geq 2$ we split the primed sum in~\eqref{pij_eq}
into three parts:  firstly, the
sum over $j>0$ with $i=0$ fixed, secondly the
sum over $i<0$ with $j=0$ fixed, and thirdly
the sum over $i$ and $j$ both not equal to 0.
This will allow us to apply the corresponding
bounds of Lemma~\ref{num_geod_est_lem}.

We apply Lemma~\ref{geod_prob_est_lem} 
with the following choices of
$r,r'$:  $r=2\gb-2$ if $i\neq0$,
$r=2\gb-4$ if $i=0$, 
$r'=2\gb$ if $j\neq0$, and
$r'=2\gb-2$ if $j=0$.  By 
Lemma~\ref{l:ximoment}, and since $\E(\xi^\gb)<\infty$,
the expectations in Lemma \ref{geod_prob_est_lem} are finite for these choices and may
be bounded by the same constant.  On applying 
Lemma~\ref{num_geod_est_lem} we 
therefore find that (for $n\geq 2$) 
there is a $c>0$ such that
\begin{equation}
\begin{split}
\E( A^{(n)}_k \mid M^\star)
\leq c\sum_{m=1}^n\Big[&\frac{n^2}
{(m-1+\gd_{m,1})^{2\gb-2}(n-m+\gd_{m,n})^{2\gb}}\\
&+\frac{1+(n-m)}
{(m-1+\gd_{m,1})^{2\gb-4}(n-m+\gd_{m,n})^{2\gb}}\\
&+\frac{1+m}{(m-1+\gd_{m,1})^{2\gb-2}(n-m+\gd_{m,n})^{2\gb-2}}\Big].
\end{split}
\end{equation}
Note that for $a,b>1$,
\begin{equation}
\sum_{m=1}^n\frac{1}{(m-1+\gd_{m,1})^a(n-m+\gd_{m,n})^b}
\leq c\Big(\frac{1}{n^a}+\frac{1}{n^b}\Big)
\end{equation}
for some $c>0$. 
We deduce that, for some constant $c'>0$,
\begin{equation}
\E(A^{(n)}_k\mid L_R)\leq c'\frac{1}{n^{2\gb-4}},
\end{equation}
and hence in~\eqref{eq:A_sum1}
\[
\sum_{n\geq1}n^q\E(A^{(n)}_k\mid L_R)\leq
c'\Big(1+\sum_{n\geq2}
\frac{1}{n^{2\gb-4-q}}\Big).
\]
This is finite for $0<q<2\gb-5$, as required.
The argument for~\eqref{eq:A_sum2}, when
$L_R=0$, is the same using the remark
immediately below Lemma~\ref{num_geod_est_lem}.
\end{proof}

\begin{proof}[Proof of Lemma~\ref{geod_prob_est_lem}]
Writing $z=\langle s_i,s_j\rangle$, 
note that 
\begin{equation}
\ell(z)=\ell(s_i)-m,\mbox{ and thus }
\ell(s_j)=\ell(s_i)-2m+n.
\end{equation}
Note that the $(m-1)$st successor 
$\sigma^{m-1}(s_i)$ of 
$s_i$ in $M^\star$ appears before $s_j$
in the contour sequence of the mobile 
$\vartheta$.
For there
to be an edge from $u\in\mathcal{C}_i$ to 
$v\in\mathcal{C}_j$ it cannot
be the case that the label of $\sigma^{m-1}(s_i)$ is
strictly smaller than that of $u$.  Thus
\begin{equation}
\ell(u)\leq \ell(\sigma^{m-1}(s_i))=\ell(s_i)-m+1.
\end{equation}
Furthermore, if $uv\in M$ then
\[
\ell(v)=\ell(u)-1\leq \ell(s_i)-m=\ell(s_j)+m-n,
\]
so that 
$\Delta\ell_j\geq\ell(s_j)-\ell(v)\geq n-m$.
Thus, using that $\Dec_i$ and 
$\Dec_j$ are independent 
of $M^\star$ and of each other,
\begin{equation}
\begin{split}
\E[\#\{u\in \mathcal{C}_i,& v\in \mathcal{C}_j: uv\in M\}\mid M^\star]\\
&\leq \E[\indic{\Delta\ell_j\geq n-m}
\#\{u\in \mathcal{C}_i: \ell(u)\leq \ell(s_i)-m+1\}\mid M^\star]\\
&\leq \P(\Delta\ell_j\geq n-m)
\E(|\Dec_i|\indic{\Delta\ell_i\geq m-1})\\
&= \P(\Delta\ell_j\geq n-m)
\E\big[|\Dec_i|\P(\Delta\ell_i\geq m-1\mid\Dec_i)\big].
\end{split}
\end{equation}
The result now follows from Lemma~\ref{l:deltal}
and Markov's inequality.
\end{proof}

\begin{proof}[Proof of Lemma~\ref{num_geod_est_lem}]
Start by considering~\eqref{eq:geo_free}. 
We treat two cases separately, Case (1)
when the part of the geodesic intersecting the spine is directed away
from the root and Case (2) when it is directed towards the root. Write
$\Gamma_{n,m}^{(k)} = \Gamma_{n,m}^{(k,1)} + \Gamma_{n,m}^{(k,2)}$ where
$\Gamma_{n,m}^{(k,\alpha)}$ denotes the number of geodesics in case ($\alpha$), 
$\alpha \in \{1,2\}$.  See Fig.~\ref{f:drawing}. 
\begin{figure} [t]
\centerline{\scalebox{0.33}{\includegraphics{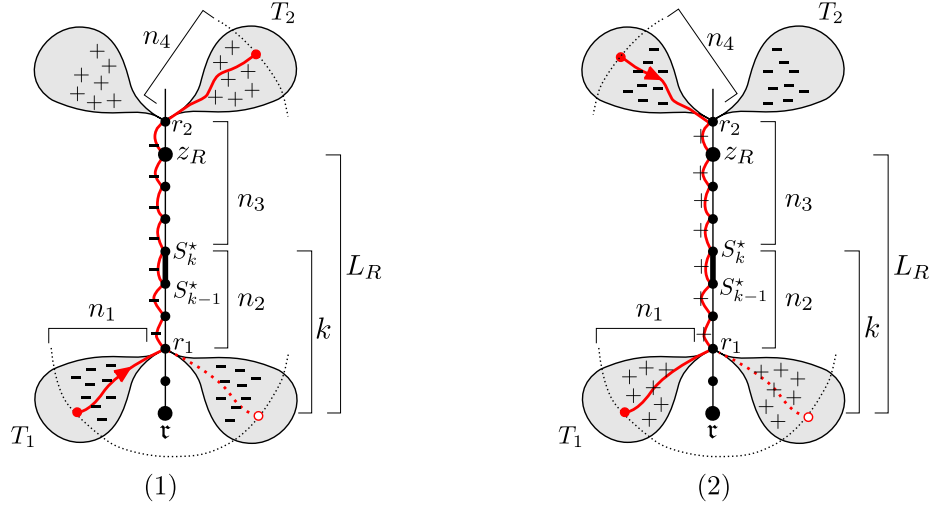}}}
\caption{The vertical path represents the spine in $M^\star$ and only the outgrowths $T_1$ and $T_2$ are drawn. In Case (1) a geodesic starts at level $n_1$ in $T_1$ and
  ends at level $n_4$ in the right hand part of $T_2$.  In Case (2) a geodesic starts in the left hand part of $T_2$
  and ends at level $n_1$ in $T_1$.   The $+$ and $-$ signs indicate whether the labels on the geodesic increase ($+$) or decrease ($-$) when one follows the direction of the geodesic.} \label{f:drawing}
\end{figure}

In Case (1), the geodesic starts in an outgrowth from a vertex, say
$r_1$, on the spine. Denote the collection of outgrowths from $r_1$ by
$T_1$. It then proceeds $m$ steps in the direction of decreasing
labels. Since it crosses the edge $S^\star_{k-1},S^\star_k$ on the
spine, then after $m$ steps it must enter an outgrowth from a vertex
on the spine, say $r_2$, which is different from $r_1$. Denote the
collection of outgrowths from $r_2$ by $T_2$. By the definition of the
direction of the geodesic it has to enter $T_2$ on the right hand side
of the spine.  We cut the geodesic into four parts. The first part is
from the starting point until it hits $r_1$, call the length of that
part $n_1$. The next part is from $r_1$ to the vertex $S^\star_k$,
call its length $n_2 \geq 1$. The third part is from $S^\star_k$ to $r_2$,
call its length $n_3$ and the final part is from $r_2$ to the end,
call its length $n_ 4$. Then
\begin {eqnarray}
n_1+n_2+n_3 = m \label{123m}\\
n_4 = n-m \label{4nm}\\
d^\star(r_1,\uiptroot) = k-n_2 \nonumber\\
d^\star(r_2,\uiptroot) = k+n_3 \nonumber
\end {eqnarray}
Conditional on $L_R$ the distributions of $T_1$ and $T_2$ 
are given by
\begin {equation}
 T_1 = (W^\star_{k-n_2}~|~W^\star_{k-n_2}(R-k+n_2)=\varnothing)
\end {equation}
and
\begin {equation}
 T_2 = \left\{\begin{array}{cl}
   (W^\star_{k+n_3}~|~W^\star_{k+n_3}(R-k-n_3)=\varnothing) 
& \text{if}~k+n_3 < L_R \\
(W^\star_{k+n_3}~|~W^\star_{k+n_3}(R-k-n_3)\neq\varnothing) 
& \text{if}~k+n_3 = L_R < R/2 \\
W^\star_{k+n_3} & \text{if}~k+n_3 = L_R = R/2 \\
& \text{or}~k+n_3 > L_R. 
  \end{array}\right.
\end {equation}
It follows that
\begin{align} \nonumber
 \mathbb{E}(\Gamma_{n,m}^{(k,1)}~|~L_R) &\leq 
\sum_{\substack{n_1+n_2+n_3=m \\ n_2\geq 1, n_3 \leq L_R-k-1}}
\mathbb{E}(Y_{n_1}~|~Y_{R-k+n_2}=0)\mathbb{E}(Y_{n_4}~|~Y_{R-k-n_3}=0) \\
& \qquad \quad + \big(\mathbb{E}(Y_{n_4}~|~Y_{R-L_R} >0)\indic{L_R < R/2}+\mathbb{E}(Y_{n_4})\indic{L_R=R/2}\big) 
 \nonumber \\
&\qquad \qquad \qquad \quad \times \sum_{\substack{n_1+n_2=m-L_R+k \nonumber \\ n_2\geq 1}} 
\mathbb{E}(Y_{n_1}~|~Y_{R-k+n_2} = 0)\\ \nonumber
& \qquad \quad+ \mathbb{E}(Y_{n_4}) 
\sum_{\substack{n_1+n_2+n_3=m \\ n_2\geq 1,  L_R-k+1\leq n_3 \leq m}} 
\mathbb{E}(Y_{n_1}~|~Y_{R-k+n_2}=0). \nonumber \\ \label{gamma1estimate}
\end{align}
The inequality sign is only due to the fact that we replace the
number of elements at level $n_4$ in the part of $T_2$ on the right
hand side of the spine by the total number of elements at level $n_4$
in $T_2$. Using Lemma~\ref{l:xixj} and recalling that $n_4 = n-m$ we get the upper bound
\begin {equation}
 \mathbb{E}(\Gamma_{n,m}^{(k,1)}~|~L_R) \leq k_1 m^2 + k_2 (n-m)m
\end {equation}
where $k_1$ and $k_2$ are positive constants.  

In Case (2) we have a
similar picture. The geodesic starts from the part of 
$T_2$ which lies to the left of the spine and ends in
$T_1$. In this case (\ref{123m}) and (\ref{4nm}) are replaced by the
conditions
\begin {eqnarray}
n_1+n_2+n_3 = n-m \\
n_4 = m, 
\end {eqnarray}
but everything else is the same. We thus get
\begin {equation}
 \mathbb{E}(\Gamma_{n,m}^{(k,2)}~|~L_R)  \leq k_3 (n-m)^2 + k_4 (n-m)m
\end {equation}
where $k_3$ and $k_4$ are positive constants. 

Now turn to~\eqref{eq:geo_i0}. In this case the origin of the geodesic
is $s_0$ and thus contains the bad decoration.  Recall that $\epsilon$
denotes the direction of the root edge in the BDG bijection.  When
$\epsilon = -1$ the root edge is directed away from $s_0$ and thus by
the definition of $\uiptroot$, $\uiptroot = s_0$. We may then go trough
the same argument as for~\eqref{eq:geo_free}, 
Case (1) except now $n_1 = 0$, $n_2= k$ and one 
does not have to take the distribution of $T_1$ into
account. Similarly, when $\epsilon = +1$, $s_0$ is the leftmost child
of $\uiptroot$ and in this case $n_1 = 1$ and $n_2 = k$. Using these
values, \eqref{gamma1estimate} may be estimated by the expression
stated in~\eqref{eq:geo_i0}.   Finally,~\eqref{eq:geo_j0}  
follows in the same way
but corresponds to Case (2) in~\eqref{eq:geo_j0} 
and thus $m$ is replaced by $n-m$.
\end {proof}

\ack{The authors would like to thank Svante Janson for
several  interesting
  discussions, 
Takis Konstantopoulos for interesting discussions about 
Lemma~\ref{kmt_lem},
and the anonymous referee for helpful suggestions. 
SÖS is grateful for hospitality at NORDITA.}

\section{Appendix}
In this section we prove Lemmas \ref{l:ximoment}, \ref{kmt_lem} and \ref{l:geosup} and conclude by collecting a few results on the modified Galton Watson process $(Y_n)_{n\geq 0}$ defined  above \eqref{lrdist1}.
\begin {proof} [Proof of Lemma \ref{l:ximoment}]
Let $Z$ be a Galton--Watson tree with offspring distribution $\xi$ and denote the total number of vertices in $Z$ by $N$. First consider the case $i\neq 0$. Then the  bijection between simply generated trees and mobiles shows that $|\Dec_i|\eqd N$.

We can interpret $N$ as a first-passage time in a random walk
with drift, according to the following
standard `depth-first-search' construction of $Z$.   
First sample the number $\xi_1$ of children of
the root.  If $\xi_1=0$ we are done (and the tree has size 1).  
Otherwise pick the `leftmost' of the children of the root
and independently of $\xi_1$ sample its number $\xi_2$ of 
children.  If $\xi_2>0$ we repeat this for the new leftmost
child, otherwise we repeat the procedure for the leftmost
of the remaining $\xi_1-1$ children of the root that have not
yet been `investigated'.  
The same procedure is repeated until the entire tree
has been constructed.
By considering the number of vertices
left to investigate after $k$ steps in this construction
we see that $N$ is the smallest value of $k$ such
that the random walk
\begin{equation}
\sum_{j=1}^k (\xi_j-1)
\end{equation}
reaches level $-1$.  (This random walk is
sometimes called the \emph{Lukasiewicz path}
of $Z$.)
Since $\E(\xi_i-1)=\kappa-1<0$
it follows from~\cite[Theorem~3.3.1]{gut:stopped}
that 
\begin{equation}\label{gut_eq}
\E(N^r)<\oo \mbox{ provided } \E(\xi^r)<\oo, 
\end{equation}
as required.

The argument for $\Dec_0$ is similar, but the bijection
from simply generated trees has a more complex result this time.
The vertices of $\Dec_0$ consist both of 
the outgrowth from the infinite degree 
vertex in $\sgt$ which is immediately to the right of the spine, as well as the vertices of 
finite degree on the spine of $\sgt$, along with their outgrowths.
The length $L$ of the
spine satisfies $\P(L=i)=\kappa^i(1-\kappa)$.
The number of outgrowths of a vertex of finite degree 
on the spine is distributed as $\tilde\xi$,
where 
\begin{equation}
\P(\tilde\xi=k)=\P(\hat\xi-1=k\mid\hat\xi<\oo)=
\frac{(k+1)\pi_{k+1}}{\kappa}
\end{equation}
(see~\eqref{hat_xi_eq}).  
Thus for all $s\geq0$
\begin{equation}\label{tilde_xi_mom_eq}
\E(\tilde\xi^s)\leq
\frac{1}{\kappa}\E(\xi^{s+1}).
\end{equation}

Let $(\tilde\xi_i)_{i\geq 1}$ be
independent copies of $\tilde\xi$, and let $N$
and $(N_{i,j})_{i,j\geq 1}$ be independent copies
of $N$ above, all independent of each other and
of $L$.  
It follows from the description above
that for all $s\geq0$
\begin{equation}\label{chi0_bound_eq}
\E(|\Dec_0|^s)\leq
\E\Big[\Big(N+\sum_{i=1}^L\sum_{j=1}^{\tilde\xi_i}N_{i,j}\Big)^s\Big].
\end{equation}
Let $s=r-1$.  If $s\geq 1$ it follows 
from~\eqref{chi0_bound_eq} 
(conditioning on $L$ and the $\tilde\xi_i$,
and using Minkowski's inequality),
that
\begin{equation}
\||\Dec_0|\|_s\leq \|N\|_s\big(1+\|L\|_s\|\tilde\xi\|_s\big),
\end{equation}
which is finite by~\eqref{gut_eq}
and~\eqref{tilde_xi_mom_eq}.
Similarly, if $0<s<1$ then by using
subadditivity and Jensen's inequality we get that
\begin{equation}
\E(|\Dec_0|^s)\leq
\E(N^s)+\E(L)\E(N)^s\E(\tilde\xi^s),
\end{equation}
which is finite by~\eqref{gut_eq}
and~\eqref{tilde_xi_mom_eq}.
\end{proof}

\begin{proof}[Proof of Lemma \ref{kmt_lem}] 
Note that $i^+(R)$ and $i^-(R)$ can be written in terms
of first passage times of the random walk 
$S_n = \sum_{i=1}^n X_i$ with the i.i.d.~jump distribution $X_i$
described at~\eqref{bridgelaw}:
\[
\begin{split}
i^+(R) &=\inf\{n\geq 0: S_n=-R\}=\inf\{n\geq 0: S_n\leq-R\},\\
i^-(R) &=\inf\{n\geq 0: S_n\geq R\}.
\end{split}
\]
For this random walk the distributions of such first
passage times can be computed explicitly
(using an exponential martingale argument).  However,
we use a more general argument
based on approximating the random walk with 
Brownian motion, and well-known results for first passage
times of the latter.  

Let $S(t)$ denote the process obtained from $S_n$
by linear interpolation between integer times.
The $X_i$ of~\eqref{bridgelaw}
have mean 0 and variance 2, so the 
Koml\'os--Major--Tusn\'ady 
theorem~\cite{kmt}
tells us that we may couple $(S(t):t\geq0)$ with a standard
Brownian motion $(W(t):t\geq0)$ in such a way that
for some constants $c,\eps>0$ and
all $T,x>0$ we have
\begin{equation}\label{kmt_eq}
\P\Big(\max_{0\leq t\leq T}|S(t)/\sqrt{2}-W(t)|>c\log T+x\Big)
\leq e^{-\eps x}.
\end{equation}
Let $K(\l)=K'\log\l$, with the constant $K'$ chosen large enough
that
\[
K(\l)R>2c\log R+(c+1/\eps)\log \l,\quad\mbox{for all }R,\l\geq1.
\]
We have that
\[
\begin{split}
\P(i^-(R)>\l R^2)&=
\P\Big(\max_{0\leq n\leq \l R^2} S_n<R\Big)\\
&\leq\P\Big(\max_{0\leq t\leq \l R^2} W(t)<(K(\l)+1/\sqrt{2})R\Big)\\
&\qquad+\P\Big(\max_{0\leq t\leq \l R^2} |W(t)-S(t)/\sqrt{2}|>K(\l)R\Big).
\end{split}
\]
The first term is at most
\[
\frac{2}{\sqrt{2\pi}}(K(\l)+1/\sqrt{2})R(\l R^2)^{-1/2}
\leq K'' \frac{\log \l}{\l^{1/2}}
\]
by standard results for Brownian motion
(by the reflection principle, the probability that
$W$ has not exceeded $a$ by time $t$ equals
$2\P(W(t)>a)$).  By~\eqref{kmt_eq}
the second term is at most
\[
\P\Big(\max_{0\leq t\leq \l R^2} |W(t)-S(t)/\sqrt{2}|>
c\log(\l R^2)+\tfrac{1}{\eps}\log\l\Big)\leq \l^{-1}.
\]
This proves the bound for $i^-(R)$.  The bound for
$i^+(R)$ is similar.
\end{proof}

\begin {proof} [Proof of Lemma \ref{l:geosup}]
We will need the following easily proved result. If $v \in \gamma(u)$ then
\begin {equation} \label{geodist}
d(u,v) = \ell(u)-\ell(v). 
\end {equation}
If $u=s_i$ and $v=s_j$ with $i<j$
define $]u,v[ ~=
\{s_{i+1},s_{i+1},\ldots,s_{j-1}\}$.

Consider first the case $v^\star = w^\star$.
Let $\lambda$ denote the minimal label in $\Dec(v)$ and
let $x$ denote the first white vertex in the contour sequence
which, firstly, occurs after the last occurrence of $v^\star$, 
and, secondly, has label $\ell(x)=\lambda-1$.  Then
$\gamma(v)$ and $\gamma(w)$ both contain $x$, so
by \eqref{geodist}
\begin{equation}\label{samedec}
\begin{split}
d(v,w)&\leq d(v,x)+d(w,x)\\
&=\ell(v)-(\lambda -1)+\ell(w)-(\lambda -1)
\leq 2+4\Delta\ell(\Dec(v)).
\end{split}
\end{equation}

Consider now the case $v^\star \neq w^\star$.
Without loss of generality we may
assume that $v^\star \wedge w^\star = v^\star$. 
First assume that $w^\star \in \geo^\star(v^\star)$ 
in which case 
$[v^\star\wedge w^\star,\langle v^\star,w^\star \rangle] = 
[v^\star, w^\star]$. Then
\begin {equation} \label{longeo}
 \ell(u) > \ell(w^\star) 
\end {equation}
for all $u\in ]v^\star,w^\star[$.  Define 
$A = \bigcup_{u\in ]v^\star,w^\star[} V^\circ(\Dec(u))$ 
and let $\lambda = \min_{u\in A} \ell(u)$, 
assuming that the minimum is attained at the vertex
$y$ (possibly among others). Then $y^\star \in ]v^\star,w^\star[$
and by (\ref{longeo}) it holds that
\begin {equation} \label{lambda1lower}
\lambda = \ell(y) \geq \ell(y^\star) - \Delta\ell(\Dec(y)) \geq \ell
(w^\star)+1 - \Delta\ell(\Dec(y)) .
\end {equation}
Consider the successor geodesics $\succg(v^\star)$ and
$\succg(w^\star)$ in $M$. They will meet at the first time at a
vertex, say $z$, with label $\ell(z) = \lambda -1$.  Then, by
\eqref{geodist} and (\ref{lambda1lower})
\begin {equation} \label{ubwstarz}
 d(w^\star,z) = \ell(w^\star) - \ell(z) \leq \Delta\ell(\Dec(y)).
\end {equation}
Finally, 
\begin{equation} \label{ubcaseii}
\begin{split}
d(v,w) &\leq 
 d(v,v^\star) + d(v^\star,z) + d(z,w^\star)+d(w^\star,w) \\ 
&=
d(v,v^\star) + \ell(v^\star)-\ell(z) + \ell(w^\star)-\ell(z) + d(w^\star,w) \\
&=
d^\star(v^\star,w^\star) +2(\ell(w^\star)-\ell(z)) 
  +d(v,v^\star) + d(w,w^\star) \\
&\leq 
d^\star(v^\star,w^\star) +  2 \Delta\ell(\Dec(y)) 
    + 4\Delta\ell(\Dec(v)) + 4\Delta\ell(\Dec(w)) + 4. 
\end{split}
\end{equation}
In the first line we  used the triangle inequality,
in the second line (\ref{geodist}), in the third line
the fact that $w^\star\in\gamma^\star(v^\star)$ implies 
$d^\star(v^\star,w^\star)=\ell(v^\star)-\ell(w^\star)$,
and in the fourth line (\ref{ubwstarz}) as well as~\eqref{samedec}. 

Finally consider the case when $w^\star \notin \geo^\star(v^\star)$.
Write $z=\langle v^\star,w^\star \rangle$.
Then $z=z^\star\in\gamma^\star(v^\star)\cap\gamma^\star(w^\star)$
and $d(v,w)\leq d(v,z)+d(w,z)$.  Applying~\eqref{ubcaseii}
to both terms and using the fact that
$d^\star(v^\star,z)+d^\star(w^\star,z)=d^\star(v^\star,w^\star)$
we arrive at the claimed bound.
\end{proof}

\subsection{The process $(Y_n)_{n\geq 0}$}
We collect here a few results which we need on the process $(Y_n)_{n\geq 0}$ which is defined above \eqref{lrdist1}.
Recall that $\pi_i = 2^{-i-1}$. Define the generating functions
\begin {equation}
 f(x) = \sum_{i=0}^\infty \pi_i x^i \quad \text{and} \quad 
g_n(x) = \sum_{i = 0}^\infty \P(Y_n = i)x^i.
\end {equation}
Clearly, $g_0(x) = x$ and assume in the following that $n\geq 1$. By standard generating function arguments and using induction one finds that
\begin {equation} \label{gnformula}
 g_n(x) = f'(f_{n-1}(x)) = \left(\frac{n-(n-1)x}{n+1-nx}\right)^2
\end {equation}
where $f_0(x) = x$ and $f_n(x) = f(f_{n-1}(x))$. 
By \eqref{gnformula} one immediately gets
\begin {equation} \label{yn0}
\P(Y_n = 0) = g_n(0) = \left(\frac{n}{n+1}\right)^2.
\end {equation}
The following Lemma is used in the proof of Lemma \ref{num_geod_est_lem}.
\begin {lemma} \label{l:xixj} 
For all $i,j\geq 1$ we have that
$\E(Y_i \mid Y_j = 0) \leq 2$
and that
$\E(Y_i \mid Y_j > 0) \leq 4 i + 4$.
\end {lemma}

\begin {proof}
We use the fact that  for each fixed $i\geq0$,
both $\E(Y_i \mid Y_j = 0)$
and $\E(Y_i \mid Y_j > 0)$ are non-decreasing in $j$.
We thus obtain an upper bound
by letting $j\rightarrow\infty$.  Since the
process $(Y_i)_{i\geq0}$  dies out
with probability 1 we have that
$\lim_{j\rightarrow\infty}\E(Y_i \mid Y_j = 0) = \E(Y_i)=g_i'(1) = 2$ by \eqref{gnformula}.
We also have that
\begin {equation}
\E(Y_i \mid Y_j > 0) = \frac{g_i'(1)-\P(Y_{j-i}=0)g_i'(\P(Y_{j-i}=0))}{1-\P(Y_{j}=0)}  \rightarrow 4i+4
\end {equation}
as $j\rightarrow \infty$ where the convergence follows from \eqref{gnformula} and \eqref{yn0}.

The monotonicity in $j$ used above may be proved as follows.
Consider first 
$E(Y_i\mid Y_j>0)$ and start with the 
case $j\leq i$.  Then we have
\[
E(Y_i\mid Y_j>0)=\frac{E(Y_i\indic{Y_j>0})}{P(Y_j>0)}
=\frac{E(Y_i)}{P(Y_j>0)},
\]
since $\{Y_j=0\}\subseteq\{Y_i=0\}$.
Clearly $P(Y_j>0)\leq P(Y_{j-1}>0)$, so the statement
holds for all $j$ up to $i$.

Now consider the case $j\geq i$.  Write 
$p_j(k)=P(Y_i=k\mid Y_j>0)$.  Then we have
\[
E(Y_i\mid Y_{j+1}>0)
=\sum_{k\geq 1} k p_{j+1}(k)
=\sum_{k\geq 0}k p_j(k) \frac{p_{j+1}(k)}{p_j(k)}
=E\big[Y_i h_j(Y_i)\mid Y_j>0\big],
\]
where
\[
h_j(k)=\frac{p_{j+1}(k)}{p_j(k)}.
\]
We may rewrite
\[
h_j(k)=\frac{P(Y_{j+1}>0\mid Y_i=k)}{P(Y_{j}>0\mid Y_i=k)}
\frac{P(Y_j>0)}{P(Y_{j+1}>0)}
=c_j\frac{1-q_{j+1-i}^k}{1-q_{j-i}^k},
\]
where $c_j={P(Y_j>0)}/{P(Y_{j+1}>0})$
and $q_r$ is the probability that an individual present at
time $i$ has no offspring at time $i+r$.
Since $q_{r+1}\geq q_r$ we have that $h_j(k)$ is
non-decreasing in $k$.  It follows from Harris' inequality that 
\[
E(Y_i\mid Y_{j+1}>0)\geq E(Y_i\mid Y_j>0)
E(h_j(Y_i)\mid Y_j>0)=E(Y_i\mid Y_j>0),
\]
as required, since
\[
E(h_j(Y_i)\mid Y_j>0)=
\sum_{k\geq 0}\frac{p_{j+1}(k)}{p_j(k)}p_j(k)=1.
\]
The argument for $E(Y_i\mid Y_j=0)$ is similar,
using the increasing function
$(q_{j+1-i}/q_{j-i})^k$ in place of $h_j(k)$.
\end {proof}


\begin {thebibliography}{99}
 
 
 
\bibitem{aldous:1986} D. Aldous and J. Pitman, {\it Tree-valued Markov chains derived
from Galton–Watson processes.} Ann. Inst. H. Poincare Probab. Statist. \textbf{34} (1998), no. 5, 637–686. 
 
\bibitem{alexander:1982} S.~Alexander and R.~Orbach, {\it Density of states on fractals: ``fractons'',} J.~Physique (Paris)~Lett. \textbf{43} (1982), 625-631.
 
 
\bibitem{angel:2003} O. Angel and O.~Schramm, {\it Uniform infinite planar triangulations.} Comm. Math. Phys., \textbf{241}, 2-3 (2003), 191--213.


\bibitem{barlow:2006} M.~T.~Barlow and T.~Kumagai, {\it Random walk on the incipient infinite cluster on trees.} Illinois J. Math.~\textbf{50}, Number 1-4 (2006), 33-65. 
\bibitem{berger:2002} N.~Berger, 
{\it Transience, recurrence and critical behaviour
for long-range percolation.} 
Comm.~Math.~Phys.~\textbf{226} (2002), 531--558.

\bibitem{benjamini:2001} I.~Benjamini and O.~Schramm, {\it Recurrence of distributional limits of finite
planar graphs.} Electron.~J.~Probab., \textbf{6} (2001), no.~23, 13 pp.
 
\bibitem{bettinelli} J.~Bettinelli,
{\it Scaling limit of random planar quadrangulations with a boundary.}
Ann. Inst. H. Poincar{\'e} Probab. Statist. (to appear)

\bibitem{bouttier:2004} J.~Bouttier, P.~Di Francesco and E.~Guitter, {\it Planar maps as labelled mobiles.} Electron.~J.~Combin.~\textbf{11} (2004) R69.

\bibitem{chassaing:2006} P.~Chassaing and B.~Durhuus, {\it Local limit of labeled trees and expected volume growth in a random quadrangulations.} Ann.~Probab.~\textbf{34}, 3 (2006) 879--917.

\bibitem{curien:2012a} N.~Curien, L.~Ménard and G.~Miermont, {\it A
    view from infinity of the uniform infinite planar
    quadrangulation.} ALEA~\textbf{10} (2013), no.~1, 45--88.
\bibitem{durhuus:2007} B. Durhuus, T. Jonsson, and J. F. Wheater, {\it The spectral dimension of generic trees.} J. Stat. Phys.~\textbf{128} (2007), no. 5, 1237–1260.
\bibitem{fuji:2008} I.~Fuji and T.~Kumagai, {\it Heat kernel estimation on the incipient infinite cluster for critical branching processes.} Proc. of the RIMS workshop on Stochastic Analysis and Applications(2008), 85-95.

\bibitem{gurel:2013} O.~Gurel-Gurevich and A.~Nachmias, {\it Recurrence of planar graph limits,} Ann.~Maths.~ \textbf{177} (2013), 761-781.

\bibitem{gut:stopped} A.~Gut,
{\it Stopped random walks,} 2nd ed.  Springer 2009.  

\bibitem{janson:2011} S.~Janson, T.~Jonsson and S.~O.~Stefansson, {\it Random trees with superexponential branching weights.} J.~Phys.~A: Math.~Theor. \textbf{44} (2011), 485002.

\bibitem{janson:2012sgt} S.~Janson, {\it Simply generated trees, conditioned Galton--Watson trees, random allocations and condensation.} Probability Surveys {\textbf{9}} (2012), 103--252.

\bibitem{janson:2012sl} S.~Janson and S.~Ö.~Stefánsson, {\it Scaling
    limits of random planar maps with a unique large face.}  (To appear.)
  arXiv:1212.5072.
\bibitem{jonsson:2011}  T.~Jonsson and  S.~\"O.~Stef\'ansson, {\it Condensation in nongeneric trees.} {Journal of Statistical Physics}, \textbf{142} (2011), no. 2,
277--313.

\bibitem{kennedy:1975} D. P. Kennedy, {\it The Galton–Watson process conditioned on the total progeny.} J. Appl. Probab. \textbf{12} (1975), 800–806.

\bibitem{kesten:1986} H.~Kesten, {\it Subdiffusive behaviour of random walk on a random cluster.} Ann.~Inst.~H.~Poincaré Probab.~Statist. \textbf{22} (1986) no.~4, 425-487.

\bibitem{kmt}  J.~Koml\'os, P.~Major and G.~Tusn\'ady ,
{\it An approximation of partial sums of independent 
RV'-s, and the sample DF. I}.  Prob. Th. Rel. Fields.
\textbf{32} (1975),  number 1--2, 111--131.

\bibitem{kozma:2009} G.~Kozma and A.~Nachmias, {\it The Alexander-Orbach conjecture holds in high dimensions.} Inventiones mathematicae, \textbf{178}, Issue 3, (2009) 635-654.

\bibitem{krikun:2005} M.~Krikun, {\it Local structure of random
  quadrangulations.} arXiv:0512304.


\bibitem{kumagai:2008} T.~Kumagai and J.~Misumi, {\it Heat kernel estimates for strongly recurrent random walk on random media.} J.~Theor.~Probab.~\textbf{21} (2008), 910-935.



\bibitem{legall:notes} J.-F.~Le Gall and G.~Miermont, {\it Scaling limits of
  random trees and planar maps.} Lecture notes for the Clay Mathematical
  Institute Summer School in Buzios, July 11 - August 7, (2010),
  arXiv:1101.4856v1.
\bibitem{legall:2007} J.-F.~Le Gall, {\it The topological structure of
  scaling limits of large planar maps.} Invent. Math., \textbf{169}, (2007) 621--670.

\bibitem{legall:2011} J.-F.~Le Gall and G. Miermont, {\it Scaling limits of random planar maps with large faces.} Ann. of Probab. \textbf{39}, 1 (2011), 1--69.

\bibitem{legall:unique} J.-F.~Le Gall, {\it Uniqueness and universality of the Brownian map.} arXiv:1105.4842.

\bibitem{lyons-peres} R.~L:yons and Y.~Peres,
{\it Probability on trees and networks}.  CUP, 2005.

\bibitem{marckert:2007} J.-F.~Marckert and G.~Miermont,  {\it Invariance principles for random bipartite planar maps.} Ann. Probab. \textbf{35} (2007), 1642–1705. 

\bibitem{menard:2013} L.~Ménard and P.~Nolin, {\it Percolation on uniform infinite planar maps}, arXiv:1302.2851.

\bibitem{miermont:unique} G.~Miermont, {\it The Brownian map is the scaling limit of uniform random plane quadrangulations.} arXiv:1104.1606.

\bibitem{miermont2006invariance}
G.~Miermont,
  {\it An invariance principle for random planar maps.}
DMTCS Proceedings \textbf{1}, 2006.


 \bibitem{schaeffer:thesis} G.~Schaeffer, {\it Conjugaison d’arbres et cartes combinatoires aléatoires.} Ph.D. thesis, Univ. Bordeaux I (1998).


\end {thebibliography}
\end{document}